\documentclass{amsart}
\usepackage{amsmath,amssymb,amsthm,textcomp,enumerate,multicol,graphicx,mathtools,
enumitem,
xcolor,tikz,tikz-cd,hyperref}
\graphicspath{ {./images/} }

\usepackage{subcaption} % successor package to subfig

\usepackage[margin=1in]{geometry}
\theoremstyle{plain}
\newtheorem{theorem}{Theorem}
\numberwithin{theorem}{section}
\newtheorem{lemma}[theorem]{Lemma}
\newtheorem{proposition}[theorem]{Proposition}
\newtheorem{corollary}[theorem]{Corollary}

\newtheorem{question}[theorem]{Question}
\theoremstyle{definition}
\newtheorem{definition}[theorem]{Definition}

\newtheorem{remark}[theorem]{Remark}
\newtheorem{example}[theorem]{Example}
\newtheorem{problem}[theorem]{Problem}

\newcommand{\ra}{\rightarrow}
\newcommand{\Z}{\mathbb{Z}}

\newcommand{\Pcal}{\mathcal{P}}

\newcommand{\Tcal}{\mathcal{T}}
\newcommand{\Dcal}{\mathcal{D}}

\newcommand{\T}{\overline{T}}

\title{Bridge multisections of knotted surfaces in $S^4$}
\author{Rom\'an Aranda}
\address{Department of Mathematics, University of Nebraska-Lincoln}
\email{jarandacuevas2@unl.edu}
\author{Carolyn Engelhardt}
\address{Department of Mathematics, University at Buffalo}
\email{cengelha@buffalo.edu}

\date{}

\begin{document}

\begin{abstract}
% This paper studies descriptions of surface-links in $S^4$ using tuples of trivial tangles called \emph{multiplane diagrams}. 
Bridge multisections are combinatorial descriptions of surface links in $S^4$ using tuples of trivial tangles. They were introduced by Islambouli, Karimi, Lambert-Cole, and Meier to study curves in rational surfaces. In this paper, we prove a uniqueness result for bridge multisections of surfaces in $S^4$: we give a complete set of moves relating to any two multiplane diagrams of the same surface. This is done by developing a surgery operation on multiplane diagrams called \emph{band surgery}. Another application of this surgery move is that any $n$--valent graph with an $n$--edge coloring is the spine of a bridge multisection for an unknotted surface. We also prove that any multisected surface in $S^4$ can be unknotted by finitely many band surgeries. 
%
% This paper studies new practical methods for representing knotted surfaces in $S^4$ diagrammatically called \emph{multiplane diagrams.} 
% In 2023, Meier, Thompson, and Zupan introduced local modifications for triplane diagrams that affect the embedded surface in a controlled way. %: by taking the connected sum with an unknotted $RP^2$ and performing 1--handle addition. 
% This article explores such operations in the context of bridge multisections. We have three applications of these moves: we show that any $n$--valent graph with an $n$--edge coloring is the spine of a bridge multisection for an unknotted surface. We prove a uniqueness result for bridge multisections: we give a complete set of moves relating any two bridge multisections of the same surface. We also prove that any multisected surface in $S^4$ can be unknotted by finitely many band surgeries. 
\end{abstract}

\maketitle

%%%%%%%%%%%%%%%%%%%%%%%%%%%%%%%
\section{Introduction}

% This paper studies descriptions of surface-links in $S^4$ using tuples of trivial tangles called \emph{multiplane diagrams}. 
Smoothly embedded surfaces in $S^4$ can be described using movies, marked vertex diagrams, broken surface diagrams, or via braid charts \cite{book_knotted_surfaces}. Some of these presentations resemble techniques for studying knots in 3--space. For instance, the 4-dimensional analog of a knot projection is a broken surface diagram, where the self-intersections of an immersed surface in 3--space are decorated with crossing information. %On the other hand, bridge splittings, often called plat projections, are alternative descriptions of links in 3--space \cite[Ch 27]{encyclopedia_knots}. 

Bridge splittings for links in 3--space, often called plat projections, are classic tools in knot theory \cite[Ch 27]{encyclopedia_knots}. 
In 2017, motivated by Gay and Kirby's theory of trisections of 4--manifolds \cite{GK}, Meier and Zupan introduced an analog of plats for surface-links in $S^4$ \cite{MZ}. \emph{Triplane diagrams} are triplets of trivial tangles such that the union of any two is a 3--dimensional bridge presentation of an unlink in $S^3$. Any surface can be described using a triplane diagram. 
Since their introduction, triplane diagrams have been used to study known surface-link invariants such as colorings \cite{Sato_coloring_tris}, Seifert solids \cite{JMMZ_solids}, triple point numbers \cite{JMMZ_classic}, and the fundamental group of the complement \cite{joseph_meridional_ranks}. Examples of new invariants obtained from triplane diagrams are the crossing number \cite{Zupan_simple}, group trisections \cite{blackwell_group_tris}, Nielsen equivalence \cite{JMMZ_classic}, and $\mathcal{L}$-invariants \cite{Blair_KT_inv,Aranda_KT_2023}. 

In 2022, Islambouli, Karimi, Lambert-Cole, and Meier introduced a generalization of triplane diagrams where the number of tangles is arbitrary \cite{islambouli2022toric}. 
A \emph{multiplane diagram} $\Tcal$ is a tuple $(T_1,\dots, T_n)$ of trivial tangles such that consecutive pairs, including the last and first tangles, form bridge presentations of unlinks in 3--space. As with triplanes, multiplane diagrams determine embeddings of surfaces in the \mbox{4--sphere.}
Invariants of triplane diagrams, such as the $\mathcal{L}$-invariant \cite{pants}, have been generalized to this setting. In Section \ref{sec:multisections}, we survey the known families of multiplane diagrams in the literature, such as 2--bridge multisections, completely decomposable diagrams, and multiplane diagrams of twist spun knots.

One advantage of multiplane diagrams is that some complexities of the multisected surface, like the bridge number, can be reduced at the expense of letting the number of tangles in a diagram be large. For instance, in Section \ref{sec:2-bridge}, we show that every unknotted surface admits a multisection with 2--bridge tangles, while the Euler characteristic of surfaces which can be represented by $b$--bridge triplane diagrams for a given $b$ is bounded. 

%

%In 2022, Islambouli, Karimi, Lambert-Cole, and Meier introduced a generalization of triplane diagrams where the number of tangles is arbitrary \cite{islambouli2022toric}. A \emph{multiplane diagram} $\Tcal$ is a tuple $(T_1,\dots, T_n)$ of trivial tangles such that consecutive pairs, including the last and first tangles, form bridge presentations of unlinks in 3--space. As with triplanes, multiplane diagrams determine embeddings of surfaces in the 4--sphere. Invariants of triplane diagrams, such as the $\mathcal{L}$-invariant \cite{pants}, have been generalized to this setting. In Section \ref{sec:multisections}, we survey the known families of multiplane diagrams in the literature, such as 2--bridge multisections, completely decomposable diagrams, and multiplanes of twist spun knots.

%The main result of this work is a uniqueness theorem for multiplane diagrams. 
The main contribution of this paper is a calculus for multiplane diagrams of smoothly embedded surfaces in $S^4$. 
Meier and Zupan introduced a complete set of triplane moves, relating any two triplane diagrams representing isotopic surfaces in $S^4$. In Section \ref{sec:multiplane_moves}, we expand the triplane moves from \cite{MZ} to a sufficient set of multiplane moves. %prove the main theorem of this work. 

\newtheorem*{thm_main1}{Theorem \ref{thm:main1}}
\begin{thm_main1}
Let $\Tcal_1$ and $\Tcal_2$ be two multiplane diagrams representing isotopic surfaces in $S^4$. There is a finite sequence of multiplane moves that turns $\Tcal_1$ into $\Tcal_2$.
\end{thm_main1}
%
% \begin{theorem}\label{thm:1}
% % A finite sequence of multiplane moves relates to any two multiplane diagrams of a given knotted surface.
% Let $\Tcal_1$ and $\Tcal_2$ be two multiplane diagrams representing isotopic surfaces in $S^4$. There is a finite sequence of multiplane moves that turns $\Tcal_1$ into $\Tcal_2$. 
% \end{theorem}

% The main technique of this work is 
We prove Theorem \ref{thm:main1} by exploring 
a more general surgery operation on multiplane diagrams called \emph{band surgery}. This process increases the bridge number by one while changing the surface link %in a controlled way: 
via isotopies, 
1–handle additions, or crosscap sums (Theorem \ref{thm:surgery_rhos_tubing}). Figure %s \ref{fig:exam_tube}, \ref{fig:exam_tube_2}, and
\ref{fig:exam_crosscap} depicts an example of band surgery on a multiplane diagram that adds two unknotted projective planes to a 2--sphere. % that turns an unknotted 2--sphere into an embedded Klein bottle in $S^4$. 
Figure \ref{fig:exam_tube_2} shows how to use band surgery to turn a 2--component unlink of 2--spheres into the spin of a trefoil knot. % via one 1--handle addition. 

%%%%%
\begin{figure}[h]
\centering
\includegraphics[width=.5\textwidth]{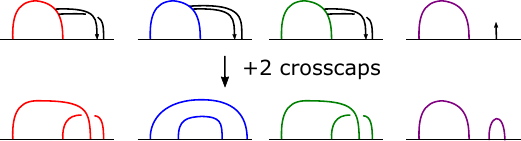}
\caption{Turning a $(1;1)$--bridge 4--section of an unknotted 2--sphere into a $(2;1)$--bridge 4--section of a Klein bottle.}
\label{fig:exam_crosscap}
\end{figure}
%%%%%

%%%%%
\begin{figure}[h]
\centering
\includegraphics[width=.5\textwidth]{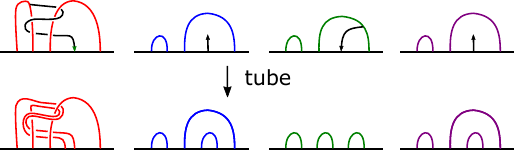}
\caption{Turning a $(2;2,2,2,2)$--bridge quadrisection of a 2--component unlink of \mbox{2--spheres} into a $(3;2,2,2,2)$--bridge quadrisection of a spun trefoil.}
\label{fig:exam_tube_2}
\end{figure}
%%%%%

In Section \ref{sec:connecting_multiplane_diagrams}, we prove any 1--handle addition on a multisected surface can be achieved by a sequence of band surgeries. Thus, band surgery is an unknotting move for multiplane diagrams, just as 1--handle addition can unknot any smooth surface in $S^4$ \cite{HK79b}.  

\newtheorem*{thm_main2}{Theorem \ref{thm:unknotting}}
\begin{thm_main2}
Let $\Tcal$ be an arbitrary multiplane diagram. There exists a finite sequence of band surgeries that turn $\Tcal$ into a multiplane diagram of an unknotted surface.  
\end{thm_main2}
%
% \begin{theorem}\label{thm:2}
% Let $\Tcal$ be an arbitrary multiplane diagram. There exists a finite sequence of band surgeries that turn $\Tcal$ into a multiplane diagram of an unknotted surface.  
% \end{theorem}

As a consequence of Theorems \ref{thm:main1} and \ref{thm:unknotting}, we obtain that any two multiplane diagrams are related by multiplane moves, band surgeries, and their inverses (see Theorem \ref{thm:uniqueness}). This motivates the problem of using band surgeries on multiplane diagrams to find bounds for the unknotting numbers of 2--knots discussed in \cite{unknotting_numbers}. Question~\ref{ques:unknotting} is motivated by the fact that the bridge index of knots in $S^3$ gives an upper bound to their tunnel number. 

\begin{question}\label{ques:unknotting}
Is there a relationship between the stabilization number of a surface $F\subset S^4$ and the bridge/patch numbers of a bridge multisection of $F$? 
\end{question}

Another application of band surgery is related to the graph isomorphism class of the spine of a bridge multisection. The spine of a multiplane diagram is the graph obtained by taking the endpoint union of its tangles. This graph is regular and equipped with edge coloring induced by the tangles in the given multiplane. In \cite{MTZ}, Meier, Thompson, and Zupan proved that every 3--valent graph with a 3--edge coloring arises as the spine of a bridge trisection for an unknotted surface. In Section \ref{sec:graphs}, we prove the analogous result for $n$--valent graphs. % bridge multisections with an arbitrary number of sectors. 

\newtheorem*{thm_main3}{Theorem \ref{thm:Tait_spines}}
\begin{thm_main3}
Let $\Gamma$ be an $n$--valent graph with an $n$--edge coloring. There is a bridge multisection $\Tcal$ for an unknotted surface in $S^4$ with 1--skeleton isomorphic to $\Gamma$ and inducing the edge coloring of $\Gamma$.
\end{thm_main3}

One way to interpret Theorem \ref{thm:Tait_spines} is that the isomorphism type of the spine of a multiplane diagram does not detect the knottedness of the underlying surface. On the other hand, by keeping track of the crossing information of each tangle of the spine, one can build distinct multiplane diagrams for the same knotted surface. In Section \ref{sec:spun_knots}, we show that any knotted surface in $S^4$ has infinitely many distinct multisections with at least four sectors. 

\newtheorem*{thm_main4}{Theorem \ref{prop:many_4_sections}}
\begin{thm_main4}
Suppose that $F\subset S^4$ admits a $b$--bridge $n$--section for some $b\geq 2$, $n\geq 4$. Then for any $b'\geq b+2$, $F$ admits infinitely many non-diffeomorphic $b'$--bridge $n$--sections. 
\end{thm_main4}

A consequence of Theorem \ref{prop:many_4_sections} is that the unknotted 2--sphere admits non-standard bridge multisections with at least four sectors, and so Otal's theorem does not hold for bridge multisections \cite{Otal}. 
%A consequence of Theorem \ref{thm:4} is that the 4--dimensional analog of Otal's theorem for unknots in the bridge position is false for bridge multisections \cite{Otal}. 
More precisely, not every multisection of the unknotted 2--sphere is perturbed.
% that the unknotted 2--sphere admits non-standard multisections with at least 4--sectors.
The existence of non-standard multisections was first observed in \cite[Thm 2.5]{islambouli2022toric} where the authors found infinitely many non-diffeomorphic 3--bridge 4--sections of the unknotted sphere. 
It is important to note that the existence of non-standard bridge trisections of the unknotted 2--sphere %, i.e., a bridge trisection version ($n=3$) of Theorem \ref{thm:4}, is still open. 
is still unknown. 

% \subsection*{Acknowledgements}
% Thank people here. 
\tableofcontents

%%%%%%%%%%%%%%%%%%%%%%%%%%%%%%%%%
\section{Background} 
We are interested in smooth embeddings of closed surfaces into $S^4$; even when not stated, smoothness will be assumed throughout this work. Following \cite{HK79b}, an embedded surface is said to be \emph{unknotted} if it can be decomposed as the distant sum\footnote{See Section~\ref{sec:sums} for the defition.} of connected sums of orientable unknotted surfaces and standard projective planes. Unknotted orientable surfaces are those that bound smoothly embedded 3--dimensional handlebodies. Equivalently, an orientable surface is unknotted if it can be isotoped to be embedded into $S^3\subset S^4$. There are two types of standard projective planes, $\Pcal_+$ and $\Pcal_-$, both differentiated by their Euler number; $e(\Pcal_\pm)=\pm 2$ \cite[Rem 2.6]{JMMZ_classic}. Figure \ref{fig:unknotted_rp2} depicts movies of $\Pcal_\pm$. 

%%%%%
\begin{figure}[h]
\centering
\includegraphics[width=.5\textwidth]{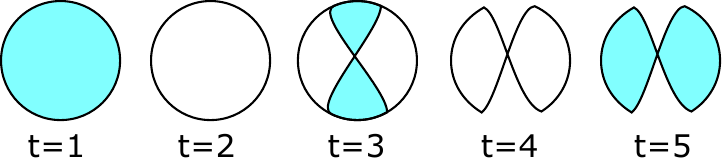}
\caption{Standard projective planes determined by different choices of crossing.}
\label{fig:unknotted_rp2}
\end{figure}
%%%%%

Given an embedded surface $F\subset S^4$, the \emph{crosscap sum} of $F$ is the new surface obtained by adding a $\Pcal_\pm$ summand to $F$. We say that $F'$ is obtained from $F$ by \emph{1--handle addition} if $F'$ is the result of replacing two small disks in $F$ with a thin tube connecting them \cite{book_knotted_surfaces}. More precisely, if $D^2\times I$ is an embedded 1--handle for $F$, so that $F\cap (D^2\times I)=D^2\times \{0,1\}$, 1--handle addition of $F$ is the new surface $F'=F\setminus \left(D^2\times\{0,1\}\right) \cup \left(S^1\times I\right)$. We refer to the core of the 1--handle $\rho =\{0\}\times I$ as the \emph{guiding arc} of the 1--handle. The property of being unknotted is preserved by 1--handle addition. 

\begin{lemma}\cite{HK79b,Kam14,BS16}\label{lem:tubing_preserves_unknot}
Let $F\subset S^4$ be a {\bf connected} unknotted surface. If $F'$ is obtained from $F$ by 1--handle addition, then $F'$ is also unknotted. 
\end{lemma}

%%%%%
\subsection{Trivial tangles}
We will work with boundary parallel $(n-2)$--balls properly embedded in $B^n$. 
In dimension four, a collection of properly embedded disks $\Dcal\subset B^4$ is called \emph{trivial} if $\Dcal$ is isotopic into the boundary of the 4–ball while fixing $\partial \Dcal$. If $|\Dcal|=c$, then we may refer to $\Dcal$ as a \emph{$c$--patch trivial disk system.} The following result implies that $c$--trivial disk systems are unique up to isotopy. 

\begin{lemma}[\cite{Liv82}]\label{lem:disk_system_uniqueness}
Suppose that $\Dcal_1$ and $\Dcal_2$ are two collections of trivial disks embedded in the 4--ball. If $\partial \Dcal_1=\partial \Dcal_2$, then $\Dcal_1$ is isotopic (rel boundary) in the 4--ball to $\Dcal_2$.
\end{lemma}

In dimension three, a \emph{trivial tangle} is a tuple $(B,T)$, or simply $T\subset B$, where $B$ is a 3--ball containing properly embedded arcs $T$ such that, fixing the endpoints of $T$, we may isotope $T$ into $\partial B$. We consider the endpoints of a tangle to be punctures of the 2--sphere $\Sigma=\partial B$. Given a trivial tangle $T\subset B$, one can find $|T|$ embedded disks in $B$ with boundary equal to the union of $T$ with arcs in $\Sigma$. Such a collection of disks is called a \emph{set of bridge disks} for $T$. For a link $L\subset S^3$, a decomposition $(S^3,L)=(B_1,T_1)\cup_{\Sigma} (B_2,T_2)$, where each pair $(B_i,T_i)$ is a trivial tangle, is called a \emph{bridge splitting}. If each tangle $T_i$ has $b$ strands, we sometimes call this decomposition a \emph{$b$--bridge splitting} of $L$. The symbol $ \T$ denotes the mirror image of a trivial tangle $T$. 

%%%%%%%%%%%%%%%%%%%%%%%%%%%%%%%%%%%%%%%%%%%%%%%
\subsection{Bands}\label{sec:bands}
Fix $T\subset B$ a properly embedded 1--manifold in a 3--ball. Classically, a band is an embedding of a rectangle that intersects a link on two opposite sides \cite{Miy86}. In this work, we will use a broader definition of bands (and band surgery) to consider rectangles that interact with the boundary of $B$ (see Figure \ref{fig:band_types}).
\begin{definition}\label{def:bands}
%More precisely, 
A \emph{band} $v$ is an embedding of a rectangle $v:I\times I\ra B$ with $(T\cup \Sigma)\cap Im(v)\subset v(I\times \{0,1\})$ and satisfying {\bf one} of the following conditions: 
\begin{enumerate}[label=(\alph*)]
\item $Im(v)\cap T=v(I\times \{0,1\})$ (this is the classical notion of band), 
\item $Im(v)\cap T=v(I\times \{0\})$ and $Im(v)\cap \Sigma=v(I\times \{1\})$, or 
\item $Im(v)\cap T=\emptyset$ and $Im(v)\cap \Sigma=v(I\times \{1\})$. 
\end{enumerate}
In each case, we define the \emph{band surgery of $T$ along $v$}, denoted by $T[v]$, to be a new tangle as in Figure \ref{fig:band_types}.
\end{definition}
Note that in cases (b) or (c), surgery along $v$ turns a $b$--stranded trivial tangle $T$ into a tangle with $(b+1)$ strands. To define band surgery in Section \ref{sec:band_surgery}, we will consider bands as in (b) and (c) where the cores are monotonic arcs (no critical point). Under these assumptions, $T[v]$ will be a trivial tangle. 
%%%%%
\begin{figure}[h]
\centering
\includegraphics[width=.4\textwidth]{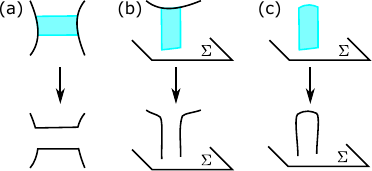}
\caption{(Top) Three kinds of bands $v$ with sides on a tangle $T$: (a) classical, (b) type 1, and (c) type 0. (Bottom) The band surgered tangle $T[v]$.}
\label{fig:band_types}
\end{figure}
%%%%%
\begin{remark}[Bands as framed arcs]\label{rem:bands_as_arcs}
In Section \ref{sec:band_surgery}, it will be useful to describe a band as in case (b) or (c) using a pair $(\alpha,f)$ where $\alpha$ is an embedded arc with framing $f$. In such cases, we will represent bands as framed arrows, where the arrow indicates the direction in which the tangle will be moved; see Figure~\ref{fig:exam_crosscap}. When the framing is not explicit, we will consider the blackboard framing as in Figure~\ref{fig:exam_tube_2}. 
\\
Suppose $T$ and $T'$ are trivial tangles with the same endpoints. If $(\alpha,f)$ represents a band for a properly embedded 1--manifold $L$, we denote by $L[\alpha,f]$ the result of $L$ after banded surgery. If $(\alpha,f)$ and $(\alpha',f')$ are framed arcs in $B$ and $B'$, respectively, with the same endpoint $*\in int(\Sigma)$, we denote by $f+f'$ the new framing on $\alpha\cup \alpha'$ provided the $f$ and $f'$ agree at $*$. This way, $(\alpha\cup \alpha',f+f')$ will represent a band for the link $T\cup \T'$ intersecting the bridge sphere in one arc. See Figure \ref{fig:framing_model} for an example of this situation. 
\end{remark}

%%%%%%%%%%%%%%%%%%%%%%%%%%%%%%%%%%%%%%%%%%%%%%%%%%%%%%%%%%%%%%%%%%%%%%
\subsubsection*{Dual bands}
We discuss further the classical notion of bands, i.e., as in (a) above. In this case, we restrict our attention to bands that are surface-framed.  That is, a band $v$ can be described as an embedded arc $\alpha$ and framing $f$ such that $\alpha$ is parallel to $\Sigma$ and $f$ is perpendicular everywhere to $\Sigma$. Suppose now that $T\subset B$ is a trivial tangle and $v$ is a collection of bands. We say that $v$ is a collection of \emph{dual bands} for the tangle $T$ if one can find a collection $D$ of bridge disks for $T$ satisfying the following conditions:
\begin{enumerate}
\item $v$ and $D$ have disjoint interiors and their union is a disjoint union of disks, and 
\item the cores of $v$ are parallel to arcs in $\Sigma$ and $v$ is surface framed. 
%
% \item the cores of $v$ are embedded arcs in $\Sigma$, 
% \item $v$ is transverse to $\Sigma$, $
% \item there is a collection $D$ of bridge disks of $T$ with interior disjoint from $v$, and
% \item every component of $D\cup v$ is contractible. 
\end{enumerate}
%%%%%
\begin{figure}[h]
\centering
\includegraphics[width=.4\textwidth]{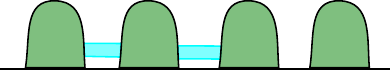}
\caption{An example of dual bands for a crossingless trivial tangle.}
\label{fig:dual_bands}
\end{figure}
%%%%%
Intuitively, after a diffeomorphism of the 3--ball, the bridge disks and bands look like Figure \ref{fig:dual_bands}. An equivalent definition of dual bands using shadow arcs instead of bridge disks can be found in \cite[Sec 3]{MZ}. Dual bands have the property that the surgered $b$--stranded tangle $T[v]$ is also trivial; see \cite[Lem 3.1]{MZ}. The following result is contained in the proof of Lemma 3.3 of \cite{MZ}.
\begin{lemma}[Lem 3.3 of \cite{MZ}]\label{lem:dual_bands}
% Dual bands exist provided $T_1\cup \T_2$ is an unlink.
Let $(B_1,T_1)$ and $(B_2,T_2)$ be two $b$--bridge trivial tangles with the same endpoints. If $T_1\cup \T_2$ is a \mbox{$c$--component} unlink in $S^3=B_1\cup \overline{B}_2$, then there exist a set $v$ of $(b-c)$ bands such that $v$ is dual to $T_1$, and $T_1[v]=T_2$. 
\end{lemma}

%%%%%%%%%%%%%%%%%%%%%%%%%%%%%%%%%%%%%%%%%%%%%%%%%%%%%%%%%%%%%%%%%%%%%%
\section{Bridge multisections}\label{sec:multisections}
We discuss a 4-dimensional analog of bridge position for links in $S^3$. Bridge multisections can be defined for surfaces in arbitrary 4-manifolds (see \cite{islambouli2022toric} or \cite{pants}). We write a simplified version for $F\subset S^4$ as this paper only studies surfaces in $S^4$. The curious reader may find interesting the work of Islambouli and Naylor on multisections of closed 4--manifolds \cite{multi}.
\begin{definition}[\cite{islambouli2022toric}]\label{def:bridge_multisection}
Let $F$ be an embedded surface in $S^4$. A $(b;c)$--bridge multisection of $F$, where $c=(c_i)_{i=1}^n$, is a decomposition 
$$(S^4,F)=(X_1,\Dcal_1)\cup \dots \cup (X_n,\Dcal_n),$$
% \bigcup_{i=1}^n \left(X_i,\Dcal_i)$ 
such that for each $i\in \mathbb{Z}_{n}$, 
\begin{enumerate}
\item $\Dcal_i=F\cap X_i$ is a $c_i$--patch trivial disk system in a 4-ball $X_i$, 
\item $(B_i,T_i)=(X_i,\Dcal_i)\cap (X_{i-1},\Dcal_{i-1})$ is a $b$--bridge trivial tangle inside a 3--ball, and 
\item $(\Sigma,\{p_1, \dots, p_{2b}\})=\bigcap_{i=1}^n (X_i, \Dcal_i)$ is a 2--sphere with $2b$--punctures. % corresponding to the endpoints of the tangles $T_i$. 
\end{enumerate}
\end{definition}

The decomposition $S^4=\bigcup_{i=1}^n X_i$ is called the \emph{0--multisection} of $S^4$. In words, a $(b;c)$--bridge multisection, also referred to as \emph{$b$--bridge $n$--section}, is a decomposition of the surface $F$ into trivial disk systems intersecting each 3--ball $B_i=X_i\cap X_{i-1}$ in $b$--bridge trivial tangles $T_i$. The surface $\Sigma$ is called the \emph{central surface} or \emph{bridge surface} of $F$ and it is often depicted as a $2b$--punctured sphere. The numbers $c_1, \dots, c_n$ are called the \emph{patch numbers}, and $b=|\Sigma\cap F|/2$ is called the \emph{bridge number} of the decomposition. When all patch numbers are equal, say to $c_1$, we will write $c$ as a number. For instance $(4;2)$--bridge multisections are $\left(4;(2,\dots,2)\right)$--bridge multisections. A $(b;c)$--bridge multisection is \emph{thin} if $c_i=b-1$ for all $i\in \mathbb{Z}_n$. Theorem \ref{thm:thin_multi} states that every surface admits a thin multisection. 

% \begin{proposition}[Thm 1.3 of \cite{MZ}]
% Every surface in $S^4$ admits a bridge trisection. 
% \end{proposition}

The union of the tangles $\bigcup_{i=1}^n (B_i,T_i)$ is called the \emph{spine} of the bridge multisection. Abstractly, the spine is a $n$--valent graph with $2b$ vertices. It follows from Definition \ref{def:bridge_multisection} that the union of consecutive tangles $L_i=T_i\cup\T_{i+1}$ is a $c_i$-component unlink in $S^3$. By Lemma \ref{lem:disk_system_uniqueness}, there is a unique way to fill the unlink with a trivial disk tangle in the 4--ball. Thus, after fixing the embeddings of the 3--balls $\{B_i\}_{i=1}^n$ in $S^4$, the tuple of tangles $(T_1,\dots, T_n)$ determines a unique embedding of $F$ up to isotopy. %determines the pair $(S^4,F)$ up to diffeomorphism. 
%In practice, we encode the spine of a bridge multisection by listing its tangles. 
%
\begin{definition}\label{def:multiplane_diagram}
A \emph{multiplane diagram}, also called \emph{$n$--plane diagram}, is an ordered tuple of trivial tangles with the same endpoints $\Tcal=(T_1,\dots, T_n)$ such that, for each $i\in \mathbb{Z}_n$, $T_i\cup \T_{i+1}$ is an unlink in $S^3$. 
\end{definition}
We often abuse notation and use $\Tcal$ to denote a multiplane diagram and a bridge multisection simultaneously. In this case, the bridge multisected surface represented by a multiplane diagram $\Tcal$ is denoted by $F_{\Tcal}$. 

An $n$--plane diagram with $n=3$ is called a \emph{triplane diagram} \cite{MZ}. They can be used to describe classical invariants of the knotted surface, such as the Euler number, the fundamental group of the complement, and the peripheral subgroup~\cite{JMMZ_classic}. One could study how these procedures extend to arbitrary multiplane diagrams. For instance, the natural extension of the formula for the Euler number of a triplane diagram also holds for $n$--planes; the proof is the same as in \cite{JMMZ_classic}. 

\begin{lemma}[Cor 3.8 of \cite{JMMZ_classic}]\label{lem:euler} 
Let $\Tcal=(T_1,\dots, T_n)$ be an $n$--plane diagram of a surface $F\subset S^4$. Let $w_i$ be the writhe of the diagram $T_i\cup\T_{i+1}$, where $i\in \Z_n$. Then the Euler number of $F$ is given by $e(F)=\sum_{i=1}^n w_i$. 
\end{lemma}

Two bridge multisections $(S^4,F)=\bigcup_{i=1}^n(X_i,\Dcal_i)$ and $(S^4,F')=\bigcup_{i=1}^n(X'_i,\Dcal'_i)$ are said to be \emph{diffeomorphic} if there is a diffeomorphism $\varphi:(S^4,F)\ra (S^4,F')$ satisfying $\varphi(X_i,\Dcal_i)=(X_{\sigma (i)},\Dcal_{\sigma (i)})$ for all $i\in \mathbb{Z}_n$, for some permutation $\sigma$. If $\varphi$ is isotopic to the identity on $S^4$, we say that the bridge multisections are \emph{isotopic}. One way to obstruct multiplane diagrams from giving diffeomorphic bridge multisections is by looking at the links obtained by gluing two non-adjacent tangles. This approach was exploited in \cite[Thm 2.5]{islambouli2022toric} to build non-standard 4-sections of the unknotted sphere. %for multisections of smooth 4-manifolds with at least four sectors. 
Given a multiplane diagram $\Tcal$, a \emph{cross-section} of $\Tcal$ is a link in $S^3$ isotopic to $T_i\cup \T_j$ for some $i,j\in \mathbb{Z}_n$. The proof of Lemma~\ref{lem:crosssections_diffeo} is left to the reader. 

\begin{lemma}\label{lem:crosssections_diffeo}
Let $\Tcal$ and $\Tcal'$ be two multiplane diagrams. Assume that the induced bridge multisections of $F_{\Tcal}$ and $F_{\Tcal'}$ are diffeomorphic. Then, after possibly permuting the tangles in $\Tcal'$, the $ij$--cross-sections of $\Tcal$ and $\Tcal'$ differ by a diffeomorphism of $S^3$.
\end{lemma}

We end this section by surveying known constructions of bridge multisections. 

%%%%%%%%%%%%%%%%%%%%%%%%%%%%%%%%%%%%%%%%%%%%%%%%%%
\subsection{2--bridge multisections}\label{sec:2-bridge}
Only the unknotted 2--sphere admits a one--bridge multisection. In this section, we aim to determine all two--bridge multisected surfaces in $S^4$ (Thm \ref{thm:class_2-bridge}). This argument follows the same outline as the work in \cite[Sec 5.3]{islambouli2022toric}. We include a brief commentary for completeness. 

Trivial tangles with two bridges are often called rational tangles as they are in correspondence with the extended rational numbers $\mathbb{Q}\cup\{\frac{1}{0}\}$ \cite{kauffman_lambropoulou}. For a pair of extended rational numbers, let $i\left(\frac{p_1}{q_1},\frac{p_2}{q_2}\right)$ be equal to $|p_1q_2-p_2q_1|$. The following result is a consequence of Conway's classification of 2--bridge links.

\begin{lemma}[\cite{conway_enumeration}]\label{lem:rational_tangles}
Let $T_1$ and $T_2$ be two rational tangles with associated fractions $\frac{p_1}{q_1}$ and $\frac{p_2}{q_2}$, respectively. 
\begin{enumerate}
\item $T_1\cup \T_2$ is a 2--component unlink if and only if $\frac{p_1}{q_1}=\frac{p_2}{q_2}$. 
\item $T_1\cup \T_2$ is an unknot if and only if $i\left(\frac{p_1}{q_1},\frac{p_2}{q_2}\right)=1$.%|p_1q_2-p_2q_1|=1$. 
\end{enumerate}
\end{lemma}

By Lemma \ref{lem:rational_tangles}, a 2--bridge multiplane diagram $\Tcal$ can be represented by a finite sequence of extended rational numbers $(\alpha_1,\dots, \alpha_n)$ where $i(\alpha_i,\alpha_{i+1})=1$ or $\alpha_i=\alpha_{i+1}$ for each $i=1,\dots, n$. Sequences of this form are loops (with possibly constant edges) in the Farey graph (see \cite{Hatcher_topology_numbers} for more details).

\begin{lemma}[Prop 5.10 of \cite{islambouli2022toric}]\label{prop:Farey_triangles}
Every loop in the Farey graph with $n\geq 3$ is conjugate to a loop $\alpha=(\alpha_1, \dots, \alpha_n)$ such that $(\alpha_1,\alpha_2,\alpha_3)$ has one of the following forms: 
$\left(\frac{0}{1},\frac{1}{0},\frac{1}{1}\right)$, $\left(\frac{0}{1},\frac{1}{0},\frac{-1}{1}\right)$, or $\left(\frac{0}{1},\frac{1}{0},\frac{0}{1}\right)$.
\end{lemma}

The triplets of fractions in Lemma \ref{prop:Farey_triangles} correspond to the three triplane diagrams with bridge number two representing connected surfaces (see \cite[Sec 4.3]{MZ}).

\begin{theorem}\label{thm:class_2-bridge}
Suppose that $\Tcal$ is a 2--bridge multiplane diagram. Then $F_\Tcal$ is either a 2--component unlink of 2--spheres or a connected unknotted surface. 
\end{theorem}
\begin{proof}
We will proceed by induction in the number of sectors of $\Tcal$. 
By \cite[Sec 4.3]{MZ}, the result holds for $n$--sections with $n= 3$. Assume that the result holds for some $n\geq 3$ and let $\Tcal$ be a 2--bridge $(n+1)$--section. If $F_\Tcal$ is disconnected, then all the tangles of $\Tcal$ must be the same and $F$ is a 2--component unlink of 2--spheres. We assume that $F_\Tcal$ is connected and the $\Tcal$ is not a constant sequence of tangles. If $\Tcal$ has a pair of consecutive tangles with $T_i=T_{i+1}$, then the disks $\Dcal_i$ between them are isotopic to product disks. Thus, erasing $T_i$ from $\Tcal$ would not affect the underlying surface. By the inductive hypothesis, the shorter tuple represents an unknotted surface and so $F_\Tcal$ is unknotted. 

Assume that no consecutive tangles of $\Tcal$ are equal. % and represent $\Tcal$ using an $n$--tuple of extended rational numbers. 
By Lemma \ref{prop:Farey_triangles}, after a cyclic permutation of $\Tcal$, we can choose a tuple of extended rational numbers representing $\Tcal$ such that the first two numbers are $\alpha_1=0/1$ and $\alpha_2=1/0$ and $\alpha_3$ is $\pm 1$ or $0/1$. In these cases, Lemma~\ref{lem:rational_tangles} implies that $T_3\cup \T_1$ is an unlink and the tuple $\Tcal_0=(T_1,T_2,T_3)$ is a triplane diagram. By the same reason, removing $T_2$ from $\Tcal$ still yields an $n$--plane diagram $\Tcal'=(T_1, T_3, T_4, \dots, T_{n+1})$. By the inductive hypothesis, $\Tcal_0$ and $\Tcal'$ represent unknotted surfaces in \mbox{$S^4$.} If $\alpha_3$ is $\pm 1$, $T_3\cup \T_1$ is an unknot and the 3--sphere $B_3\cup \overline{B}_1$ is a connect summing sphere for $F_\Tcal$ with connect summands the surfaces $F_{\Tcal_0}$ and $F_{\Tcal'}$. Hence, $F_\Tcal=F_{\Tcal_0}\#F_{\Tcal'}$ is also unknotted. 
If $\alpha_3$ is $0/1$, the link $T_3\cup \T_1$ is a 2--component unlink and the triplane diagram $\Tcal_0=(T_1,T_2,T_3)$ represents an unknotted 2--sphere. This way, the action of inserting $T_2$ into $\Tcal'$ corresponds to replacing two disks of $F_{\Tcal'}$ with one annulus. In words, $F_{\Tcal}$ is obtained by adding a 1--handle to $F_{\Tcal'}$. As consecutive tangles of $\Tcal$ are distinct, $\Tcal'$ is not a constant sequence of tangles, and so $F_{\Tcal'}$ is a connected unknotted surface. On the other hand, Lemma \ref{lem:tubing_preserves_unknot} states that adding a 1--handle to a connected unknotted surface yields an unknotted surface. Hence $F_\Tcal$ is also unknotted. 
\end{proof}

%%%%%%%%%%%%%%%%%%%%%%%%%%%%%%%%%%%%%%%%%%%%%%%%%%
\subsection{Crossingless multisections}
Besides using the bridge number, one can filter multiplane diagrams by crossing number. %use the crossing numbers of the tangles in a multiplane diagram to tabulate knotted surfaces. 
Meier and Zupan showed that triplane diagrams with no crossings are those of unknotted surfaces \cite[Prop 4.4]{MZ}. The authors of \cite{Zupan_simple} continued this program by estimating the crossing numbers of the surfaces from Yoshikawa's table \cite{Yoshi}. They also showed that the crossing number of the unknotted non-orientable surface $\mathcal{P}^{m,n}$ is equal to \mbox{$\max\{1,|m-n|\}$.} For arbitrary multiplane diagrams, only the zero-crossing case has been dealt with. 
\begin{proposition}[Prop 4.4 of \cite{MZ}]\label{prop:zero_crosing}
An orientable knotted surface is unknotted if and only if it admits a multiplane diagram with no crossings. 
\end{proposition}
%The proof of Proposition \ref{prop:zero_crosing} is the same as that of \cite[Prop 4.4]{MZ}. 
Proposition~\ref{prop:zero_crosing} was proven for triplane diagrams in \cite[Prop 4.4]{MZ}. That said, the same proof works for arbitrary $n$--plane diagrams. 
Below, we ask for a strengthening of this proposition. If the answer to Question \ref{thm:low_crosing} is positive, one may expect a proof using the technology in Section \ref{sec:perturbation_rev}. 
\begin{question}\label{thm:low_crosing}
Let $\Tcal$ be a multiplane diagram such that each tangle of $\Tcal$ has at most one crossing. Is the surface represented by $\Tcal$ an unlink of unknotted surfaces? 
\end{question}

%%%%%%%%%%%%%%%%%%%%%%%%%%%%%%%%%%%%%%%%%%%%%%%%%%%
\subsection{New from old}\label{sec:sums}
Given two surfaces $F_1$ and $F_2$, their \emph{distant sum} $F_1\sqcup F_2$ is defined by taking the connected sum of their ambient spaces using small 4--ball neighborhoods of points disjoint from $F_1$ and $F_2$. If one instead chooses neighborhoods of points in $F_1$ and $F_2$, the resulting surface in $S^4$ is $F_1\#F_2$ the \emph{connected sum} of $F_1$ and $F_2$. If the surfaces are equipped with bridge $n$--sections $\Tcal_1$ and $\Tcal_2$, then one can choose the connecting sum points between the ambient $S^4$s to lie in the central surfaces of $\Tcal_1$ and $\Tcal_2$. This will yield $(b_1+b_2)$--bridge $n$--sections for $F_1\sqcup F_2$ and $(b_1+b_2-1)$--bridge $n$--sections for $F_1\# F_2$. At the level of multiplane diagrams, this translates as taking the distant sum or the pointed-connected sum of the respective tangles. See \cite[Sec 2.2]{MZ} %\cite[Sec 4]{Blair_KT_inv} 
and \cite[Sec 2.3]{pants} for more details. 

%%%%%%%%%%%%%%%%%%%%%%%%%%%%%%%%%%%%%%%%%%%%%%%%%%%%%%%%
\subsection{Completely decomposable diagrams}
Triplane diagrams with at most three bridges have been classified \cite[Thm 1.8]{MZ}. A completely decomposable bridge trisection is one built out of low-bridge trisections. More precisely, a bridge trisection is \emph{completely decomposable} if it can be written as the distant sum of connected sums of bridge trisections with at most three bridges. These bridge trisections will play a role when defining the merging move in Section \ref{sec:split_merge}. The following propositions are sufficient conditions for a bridge trisection to be completely decomposable. In \cite[Thm 3.5]{Aranda_KT_2023}, it was proven that bridge trisections with $\mathcal{L}^*(\Tcal)\leq2$ are also completely decomposable. 

\begin{proposition}[Prop 4.1 of \cite{MZ}]\label{prop:completely_decomp}
Let $\Tcal$ be a $(b;c_1,c_2,c_3)$--bridge trisection with $c_i=b$ for some $i=1,2,3$. Then, $\Tcal$ is completely decomposable, and the underlying surface-link is the unlink of $\min_i\{c_i\}$ 2--spheres. 
\end{proposition}

\begin{proposition}[Thm 5.2 of \cite{JMMZ_solids}]\label{prop:completely_decomp_2}
Let $\Tcal$ be a $(b;c_1,c_2,c_3)$--bridge trisection with $c_i=b-1$ for some $i=1,2,3$. Then, $\Tcal$ is completely decomposable, and the underlying surface-link is either the unlink of $\min_i\{c_i\}$ 2--spheres or the unlink of $\min_i\{c_i\}$ 2--spheres and one projective plane, depending on whether $|c_{i-1}-c_{i+1}|=1$ or $c_{i-1}=c_{i+1}$. 
\end{proposition}

Less is known about completely decomposable bridge multisections with at least four sectors. Given $n\geq 4$, a $b$--bridge $n$--section is \emph{completely decomposable} if it can be written as the distant sum of the connected sum of multisections with bridge number at most two. In \cite[Thm 6.2]{pants}, the authors used the $\mathcal{L}^*$--invariant to give a criterion for completely decomposable multisections. 

For the invested reader, the collection $\{\Tcal^p_{m,n}: p\geq 3, m\geq 1, n\geq 4\}$ in the proof of Theorem~\ref{prop:many_4_sections} contains non-completely decomposable $n$--sections of the unknotted 2--sphere for all $n\geq 4$. This claim is true as the cross-sections of completely decomposable multisections are distant sums of unlinks and connected sums of 2--bridge links. 

%%%%%%%%%%%%%%%%%%%%%%%%%%
\subsection{Multisections of twist spun knots} \label{sec:spun_knots}
In this subsection, we will use the concept of \emph{bridge bisection} of a surface with boundary properly embedded in $B^4$, which we define now. Suppose $T_a$, $T_b$, and $T_c$ are $b$--bridge trivial tangles with the same endpoints such that $T_a\cup \T_b$ and $T_b\cup \T_c$ are unlinks but $T_a\cup \T_c$ may be knotted. Let $(B^4, \Dcal_{ab})$ and $(B^4,\Dcal_{bc})$ be trivial disks bounded by $T_a\cup \T_b$ and $T_b\cup \T_c$. A \emph{bridge bisection} for a surface $F\subset B^4$ is a decomposition $(B^4,\Dcal_{ab})\cup_{(B^3,T_b)} (B^4,\Dcal_{bc})$. In this case, we say that $(T_a,T_b,T_c)$ is a \emph{bisection diagram} for $F$. Examples of bridge bisections for surfaces with boundary can be built by removing one 4-dimensional sector from a bridge trisection of a closed surface. The proof of Theorem 1.3 in \cite{MZ} can be used to show that every ribbon surface in $B^4$ admits a bridge bisection; we leave it as an exercise for the reader to verify this claim.

Let $\beta$ be a $(2b-1)$--strand braid and $m$ a fixed integer. We define the link $K(\beta)$ in $b$--bridge position, and the $(2b-1)$--strand tangles $T_a(\beta)$, $T_b(\beta,m)$, and $T_c(\beta)$ as in Figure \ref{fig:spun_knot_4sec}. In \cite[Sec 4.3]{pants}, the authors explained how $\left(T_a(\beta), T_b(\beta,m), T_c(\beta), T_b(\beta,0)\right)$ is a multiplane diagram for a $(2b-1,b)$--bridge 4--section of the $m$--twist spun of $K(\beta)$. In fact, $\left(T_a(\beta), T_b(\beta,m), T_c(\beta)\right)$ is a bridge bisection diagram for the ribbon disk obtained by $m$--twisting the half spin of $K(\beta)$. The following is a consequence of the observation that the tuple $\left(T_c(\beta), T_b(\beta,m), T_a(\beta)\right)$ represents the same ribbon disk in the opposite orientation.
%
%%%%%
\begin{figure}[h]
\centering
\includegraphics[width=.7\textwidth]{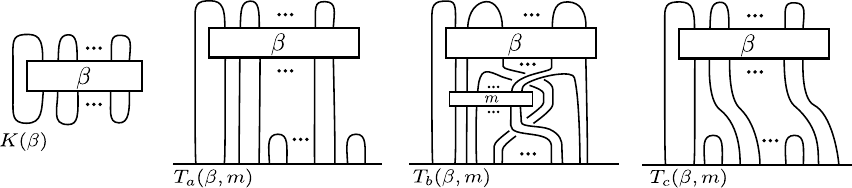}
\caption{Tangles determining multiplane diagrams for twist spun knots.}
\label{fig:spun_knot_4sec}
\end{figure}
%%%%%
%
\begin{proposition}\label{prop:multisec_spun}
Let $\beta\in B_{2b-1}$, $K=K(\beta)$ as above and let $(m_1,\dots, m_{2l})\in \Z^{2l}$ be a vector with alternating sum $m=\sum (-1)^i m_i$. Then the tuple 
\[
\left(T_a(\beta), T_b(\beta,m_1), T_c(\beta), T_b(\beta,m_2), \dots, 
T_a(\beta), T_b(\beta,m_{2l-1}), T_c(\beta), T_b(\beta,m_{2l})\right )
\]
is a $(2b-1;b)$--multiplane diagram for the $m$--twist spin of $K$. We denote such tuple as $\Tcal\left(\beta; (m_1,\dots, m_{2l})\right)$.
\end{proposition}

Using Proposition \ref{prop:multisec_spun}, we can build many multiplane diagrams of the same surface by choosing distinct alternating sums of a fixed integer. For instance, the set 
\[
\{\Tcal\left(\beta;(0,m,0,k-m)\right): m\geq 0\}
\]
contains infinitely many non-diffeomorphic 8--sections of the same $k$--twist spun knot. One can see this by comparing the cross-sections of such multiplane diagrams. The following result exploits this idea to build distinct multiplane diagrams of the same surface. Notice that Theorem \ref{prop:many_4_sections} is stated for multisections with at least four sectors. The trisection version ($n=3$) of this result is still open. 

\begin{theorem}\label{prop:many_4_sections}
Suppose that $F\subset S^4$ admits a $b$--bridge $n$--section for some $b\geq 2$, $n\geq 4$. Then for any $b'\geq b+2$, $F$ admits infinitely many non-diffeomorphic $b'$--bridge $n$--sections. 
\end{theorem}
\begin{proof}
We will find bridge multisections $\{\Tcal^p_{m,n}:p\geq 3,m\geq 1, n\geq 4\}$ satisfying the following properties: (1) each $\Tcal^p_{m,n}$ is a $p$--bridge $n$--section for an unknotted 2--sphere and, (2) the cross-sections of $\Tcal^p_{m,n}$ are the split union of unlinks and at most one $(3,3m)$--torus link. This way, if $\Tcal$ is a $b$--bridge $n$--section for $F$ for some $n\geq 4$, then $\mathcal{F}^p_n=\{\Tcal\# \Tcal^p_{m,n}:m\geq 1\}$ will be a collection of $(b+p-1)$--bridge $n$--sections representing $F$. 

We explain how to build $\Tcal^3_{m,n}$. Let $\sigma_1,\sigma_2\in B_3$ be the standard 3--braid generators and let $\vec V_m$ be the vector $(m,0,0,\dots, 0)\in \Z^{2l}$. As $K(\sigma_1)$ is a $2$--bridge unknot, Proposition \ref{prop:multisec_spun} states that $\Tcal^3_{m,4l}=\Tcal(\sigma_1,\vec V_m)$ is a $3$--bridge $4l$--section of the unknotted 2--sphere. By construction, all the cross-sections of $\Tcal^3_{m,4l}$ are of the form $T_a(\sigma_1)\cup \overline{T_b(\sigma_1,m)}$, $T_a(\sigma_1)\cup \overline{T_b(\sigma_1,0)}$, $T_a(\sigma_1)\cup \overline{T_c(\sigma_1)}$, $T_c(\sigma_1)\cup \overline{T_b(\sigma_1,m)}$, $T_c(\sigma_1)\cup \overline{T_b(\sigma_1,0)}$, and $T_b(\sigma_1,m)\cup \overline{T_b(\sigma_1,0)}$. Observe that the first five links are unlinks, and the last is the $(3,3m)$--torus link. To build $\Tcal^3_{m,n}$ for $n$ not a multiple of four, we can add trivial sectors by duplicating the desired amount of tangles in our multiplane diagram $\Tcal^3_{m,4l}$.

To build $\Tcal^p_{m,n}$ for arbitrary $p\geq 3$, one may use a move of multiplane diagrams that increases the bridge index by one without changing the underlying surface. In Section \ref{sec:perturbation}, we introduce such a move called a multiple-sector perturbation. A particular case of this is the 0--sector perturbation which does not change the link types appearing as a subset of the cross-sections of a given multiplane diagram (see Remark \ref{rem:0_pert}). In particular, we can define $\Tcal^{p}_{m,n}$ inductively to be a 0--sector perturbation of $\Tcal^{p-1}_{m,n}$ for each $p\geq 4$. Hence, the set $\{\Tcal^{p}_{m,n}:p\geq 3,m\geq 1, n\geq 4\}$ will still satisfy conditions (1) and (2). 

We now explain why the collection $\mathcal{F}^p_n$ contains infinitely many pairwise non-diffeomorphic $n$--sections. Notice first that $(3,3m)$--torus links yield non-isotopic links for distinct choices of integers $m\geq 1$ and that the cross-sections of $\Tcal\# \Tcal^p_{m,n}$ are connected sums of cross-sections of $\Tcal$ and $\Tcal^p_{m,n}$. These observations, together with the uniqueness of prime decompositions for links \cite{prime_links}, imply the existence of an infinite set $\mathcal{G}^p_n\subset \mathcal{F}^p_n$ with the following property: for each pair of distinct $x,y\in \mathcal{G}^p_n$, at least one cross-section of $x$ is not isotopic to any cross-section of $y$. Hence, by Lemma \ref{lem:crosssections_diffeo}, $\mathcal{G}^p_n$ is an infinite set of pairwise non-diffeomorphic $n$--sections for $F$. 
\end{proof}

% {\color{red}Talk about the effect of perturbation on cross-sections}
% \[
% \{\Tcal \#  U':U' \text{ is a 1-sector perturbation of } U\in \mathcal{U}_4\}
% \]

% {\color{red}Now study these multisections using alternating sums of one. All surfaces are the unknotted 2--sphere. }

% {\color{red} talk about multiplane diagrams separated by lots of perturbations}

% \begin{example}
% {\color{red} Can I show a sequence of perturbations connecting any two 1-twist spun 2--bridge knots????????????}
% \end{example}

%%%%%%%%%%%%%%%%%%%%%%%%%%%%%%%%%%%%%%%%%%%%
\section{Band surgery on multiplane diagrams}\label{sec:perturbation_rev}

In \cite{MTZ}, the authors described local modifications for triplane diagrams that affect the embedded surface by taking the connected sum with an unknotted $RP^2$ and by performing 1--handle addition. In this section, we explore these operations in the context of bridge multisections. We have three applications of these moves: Theorem \ref{thm:Tait_spines} shows that any $n$--regular Tait graph can appear as the spine of a bridge multisection for an unknotted surface, Theorem \ref{thm:main1} is a uniqueness result for multiplane diagrams yielding the same surface, and Theorem \ref{thm:uniqueness} is a uniqueness result for all multiplane diagrams. %, and a proof for Theorem \ref{thm:low_crosing} that states that multiplane diagrams with few crossings represent unknotted surfaces. 

%%%%%%%%%%%%%%%%%%%%%%%%%%%%%%%%%%%%%%%%%%%%%%
\subsection{Neighborhoods of a point}\label{sec:neighborhoods}

One can build new multiplane diagrams from old ones by performing local modifications around points of the central surface $\Sigma$. For instance, a multiplane diagram for the connected sum of two multisected surfaces can be obtained by taking the connected sum of each pair of tangles along the same puncture (see Section \ref{sec:sums}). 
% The elementary perturbation of a triplane diagram follows the same principle: to replace a certain neighborhood of an interior point of the central surface $*\in int(\Sigma)$ by a suitable properly embedded surface in the 4-ball. 
Similarly, the elementary perturbation of a triplane diagram involves replacing a certain neighborhood of an interior point of the central surface $*\in int(\Sigma)$ by a suitable properly embedded surface in the 4-ball.
Below we describe other ways in which a 4-ball neighborhood of $*\in int(\Sigma)$ can interact with a given multisected surface. 

The goal is to describe 1--handle additions and crosscap sums using framed arcs in each tangle of a multiplane diagram. Definition \ref{def:admissible_arcs} studies two kinds of arcs: those that intersect a tangle (\emph{type 1}) and those that don't (\emph{type 0}). For this, recall the discussion on bands from Section~\ref{sec:bands}; specifically panels (b) and (c) in Figure~\ref{fig:band_types}. Lemma \ref{lem:disk_from_arcs} explains how the arc data determines a disk neighborhood of $*\in int(\Sigma)$. This is important as we want to understand what part of the multisected surface we are modifying.

\begin{definition}\label{def:admissible_arcs}
Let $(S^4,F)=\bigcup_{i=1}^n (X_i, \Dcal_i)$ be a bridge multisection with multiplane diagram $\Tcal$ and let $*\in int(\Sigma)$ be an interior point of the central surface. A collection of arcs $\rho=\{\rho_i\}_{i=1}^{n}$ is called \emph{admissible} if it satisfies the following properties: 
\begin{enumerate}
\item each $\rho_i$ is an increasing arc in $B_i$ with one endpoint equal to $*$ and interior disjoint from $T_i$,  \label{prop:increasing}
\item $\partial \rho_i\cap T_i$ is at most one point; if $\rho_i\cap T_i=\emptyset$ we say that $\rho_i$ is \emph{type 0} and \emph{type 1} otherwise, 
% \item at least one arc $\rho_i$ is of type 0, 
\item\label{condition:rho_same_components} if both $\rho_i$ and $\rho_{i+1}$ are type 1 and $\rho_i\cup \rho_{i+1}$ connects the same component of $L_i$, then there exists an embedded 2--sphere $S\subset B_i\cup_\Sigma \overline{B}_{i+1}$ satisfying %the conditions of Lemma \ref{lem:existence_Delta}, 
\begin{enumerate}
\item $\rho_i\cup \rho_{i+1}$ is a subset of $S$, 
\item $S\cap L_i$ is an arc $\alpha$ of $L_i$ connecting the endpoints of $\rho_i\cup \rho_{i+1}$, and
\item $S\cap \Sigma$ is a loop containing $*$ and some punctures of $\Sigma$.
\end{enumerate}
\item\label{condition:rho_distinct_components} if both $\rho_i$ and $\rho_{i+1}$ are type 1 and $\rho_i\cup \rho_{i+1}$ connects distinct components $l_i$ and $l_i'$ of $L_i$, then there are pairwise disjoint embedded 2--spheres $S$ and $S'$ in $B_i\cup \overline{B}_{i+1}$ satisfying that 
\begin{enumerate}
\item $L_i\cap S=l_i$, $L_i\cap S'=l_i'$, 
\item $S\cup S'$ intersects $\rho_i\cup \rho_{i+1}$ in the endpoints of the arc, and 
\item $S\cap \Sigma$ and $S'\cap \Sigma$ are loops containing some punctures of $\Sigma$ but not containing $*$.
\end{enumerate}
\end{enumerate}
\end{definition}

Conditions (3) and (4) above may appear a bit cumbersome. In practice, one can interpret them as follows: after some isotopies of $L_i\cup \rho_i\cup \rho_{i+1}$ fixing the bridge surface $\Sigma$, $L_i\cup \rho_i\cup \rho_{i+1}$ has a diagram for which $\alpha\cup \rho_i\cup \rho_{i+1}$ (resp. $l_i\cup l_i'\cup\rho_i\cup \rho_{i+1}$) is crossingless like in the left (resp. right) side of Figure \ref{fig:existence_Delta}. Figures \ref{fig:exam_crosscap} and \ref{fig:exam_tube} have arcs satisfying Conditions (3) and (4), respectively.  Another feature of these conditions is that one can find embedded disks $D_i\subset \partial X_i$ bounded by $L_i$ and interior disjoint from $\rho_i\cup \rho_{i+1}$. If the arcs satisfy condition (3), the disks $D_i$ can be chosen so that $\rho_i\cup \rho_{i+1}$ is isotopic to an arc in $D_i$.

%%%%%
\begin{figure}[h]
\centering
\includegraphics[width=.6\textwidth]{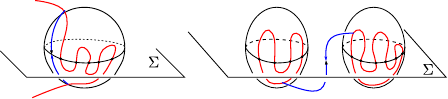}
% $\quad$
% \includegraphics[width=.3\textwidth]{images/fig_existence_Delta_b.PNG}
\caption{Model for arcs $\rho_i$ and $\rho_{i+1}$ satisfying conditions (3) and (4).}% determining a perturbation disk.}
\label{fig:existence_Delta}
\end{figure}
%%%%%

\begin{lemma}\label{lem:disk_from_arcs}
Let $\rho$ be a collection of admissible arcs for $\Tcal$. For every $i=1,\dots, n$, there is an embedded disk $\Delta_i\subset X_i$ with interior in $int(X_i)$ satisfying the following conditions; see Figure~\ref{fig:surgery_models} for a model of $\Delta_i$.

\begin{enumerate}
\item $\partial \Delta_i$ is the point union of $\partial \Delta_i \cap \partial X_i = \rho_i\cup \rho_{i+1}$ and an arc $\delta_i$ properly embedded in $X_i$.
\begin{enumerate}
\item If both $\rho_i$ and $\rho_{i+i}$ are type 1 and $\rho_i\cup \rho_{i+1}$ connects the same component of $L_i$, then $\Delta_i \cap \Dcal_i = \delta_i$. 
\item In the other cases\footnote{Some of $\rho_i$, $\rho_{i+1}$ is type 0 or both are type 1 with $\rho_i\cup\rho_{i+1}$ connecting distinct components of $L_i$}, $int(\delta_i)\cap \Dcal_i=\emptyset$ %is an arc with interior disjoint from $\Dcal_i$ 
and $\Delta_i\cap \Dcal_i=(\rho_i\cap T_i)\cup (\rho_{i+1}\cap T_{i+1})$.
\end{enumerate}
\item\label{condition:Deltai} $\Dcal_i \cup \Delta_i$ can be simultaneously be pushed into $\partial X_i$ by an isotopy that fixes $L_i\cup \rho_i\cup \rho_{i+1}$.
\end{enumerate}
Moreover, each $\Delta_i$ is unique up to isotopy in $X_i$ fixing $L_i\cup\rho_i\cup \rho_{i+1}$. 
\end{lemma}

\begin{proof}
Uniqueness of $\Dcal_i$ is the statement of Lemma \ref{lem:disk_system_uniqueness}. Conditions (1) and (2) for $\Delta_i$ above imply that the disk $\Delta_i$, depicted as a properly embedded disk in the complement of $\Dcal_i$, is a trivial disk system. As $\overline{B^4-N(\Dcal_i)}$ is diffeomorphic to a connected sum of copies of $S^1\times B^3$, Lemma 8 of \cite{MZ_4mans} implies that $\Delta_i$ is the unique disk system up to isotopy fixing $\delta_i\cup \rho_i\cup \rho_{i+1}$. Thus, any other disk $\Delta'_i$ for $\Dcal_i$ must be isotopic to $\Delta_i$ (rel boundary) in the complement of $\Dcal_i$.

To prove existence, we have cases to discuss, depending on the types of $\rho_i$ and $\rho_{i+1}$ and if $\rho_i\cup \rho_{i+1}$ connects the same component of $L_i$ or not. If one of the arcs is type 0, say $\rho_i$, we can shrink it so it lies near the point $*$ and then enlarge it so that it is equal to $\rho_{i+1}$. This process describes an isotopy of $\rho_i$ into $\rho_{i+1}$ that starts and ends in $\partial B^4$ and occurs in the interior of $B^4$ at the rest of the times; the desired disk $\Delta_i$ is the trace of the isotopy.  %This movie traces the desired disk $\Delta_i$. 

Suppose that $\rho_i$ and $\rho_{i+1}$ are both type 1 and $\rho_i\cup \rho_{i+1}$ connects distinct components of $L_i$. Condition \eqref{condition:rho_distinct_components} implies the existence of a collection of embedded disks $D_i\subset \partial X_i$ bounded by $L_i$ and interior disjoint from $\rho_i\cup \rho_{i+1}$ (see Figure \ref{fig:existence_Delta}). Let $\Delta_i$ be the disk traced by slightly pushing $\rho_i\cup \rho_{i+1}$ into the interior of $X_i$ while fixing its endpoints. As $int(D_i)\cap (\rho_i\cup \rho_{i+1})=\emptyset$ and $\Dcal_i$ is isotopic to $D_i$ (Lemma \ref{lem:disk_system_uniqueness}), then $\Delta_i$ satisfies (1) and (2) as desired. 
Suppose now that $\rho_i$, $\rho_{i+1}$ are both type 1 and $\rho_i\cup \rho_{i+1}$ connects the same component of $L_i$. Condition \ref{condition:rho_same_components} implies the existence of a collection of embedded disks $D_i\subset \partial X_i$ bounded by $L_i$ and containing the arc $\rho_i\cup \rho_{i+1}$. Then $\Delta_i$ as in the previous case will suffice. 
\end{proof}

In the following lemma, $L[\rho,f]$ denotes band surgery along an band described by a framed arc $(\rho,f)$ as in Section~\ref{sec:bands}.

\begin{lemma}\label{lem:framing_condition}
Let $\rho$ be a collection of admissible arcs for $\Tcal$. Suppose that $\rho_i$ and $\rho_{i+1}$ are both type 1. % for some $i$.%=1,\dots, n$. 
\begin{enumerate}
\item If $\rho_i\cup \rho_{i+1}$ connects the same component of $L_i$, then there exist exactly two framings, called the $\pm 1$--framings $f$, such that $L_i[\rho_i\cup\rho_{i+1},f]$ is a $c_i$--component unlink. Moreover, there exist exactly one framing, called the \emph{zero-framing} $f$, such that $L_i[\rho_i\cup\rho_{i+1},f]$ is a $(c_i+1)$--component unlink.  
\item If $\rho_i\cup \rho_{i+1}$ connects distinct components of $L_i$, then for any framing $f$, $L_i[\rho_i\cup\rho_{i+1},f]$ is a \mbox{$(c_i-1)$--component} unlink. 
\end{enumerate}
\end{lemma}

\begin{proof}
Conditions \eqref{condition:rho_same_components} and \eqref{condition:rho_distinct_components} of Definition \ref{def:admissible_arcs} ensure in each case the existence of disks $D_i\subset \partial X_i$ with boundary $L_i$ and interior disjoint from $\rho_{i}\cup \rho_{i+1}$. In particular, $\rho_i\cup \rho_{i+1}$ intersects either one or two components of the $c_i$ components of $L_i$. 
If $\rho_i\cup \rho_{i+1}$ connects the same component of $L_i$, then $L_i[\rho_i\cup\rho_{i+1},f]$ is the union of an unlink of $(c_i-2)$ unknots together with a $(2,n)$--torus link, where $n$ depends on the framing of $\rho_i\cup \rho_{i+1}$. In this case, there are exactly two framings ($n=\pm 1$) for which the torus knot is an unknot and one framing ($n=0$) for which we have a 2--component unlink. On the other hand, if $\rho_i\cup \rho_{i+1}$ connects distinct components of $L_i$, the band given by any framing $f$ becomes a connected sum band. Hence, $L_i[\rho_i\cup\rho_{i+1},f]$ is a $(c_i-1)$--component unlink. 
\end{proof}

%%%%%
\begin{figure}[h]
\centering
\includegraphics[width=.6\textwidth]{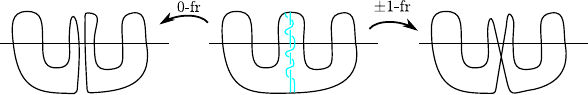}
\caption{Two kinds of framings satisfy that $L_i[\rho_i\cup \rho_{i+1},f]$ is an unlink.}
\label{fig:framing_model}
\end{figure}
%%%%%

\begin{definition}\label{def:admissible_tuple_frames}
Let $(\rho,f)=\{(\rho_i,f_i)\}_{i=1}^n$ be a tuple of framed arcs in a multiplane diagram $\Tcal$. We say that $(\rho,f)$ is \emph{admissible} if $\rho$ is a collection of admissible arcs for $\Tcal$, the framings $\{f_i\}_{i=1}^n$ agree at $*\in int(\Sigma)$, and the framing $f_i+f_{i+1}$ is zero or $\pm 1$ whenever $\rho_i\cup\rho_{i+1}$ connects the same component of $L_i$. 
\end{definition}

Note that framing on type 0 arcs does not change the isotopy class of the surgered tangle. Thus in our figures, we will only emphasize the framing on type one arcs--e.g. Figure \ref{fig:exam_crosscap}. 

%%%%%%%%%%%%%%%%%%%%%%%%%%%%%%%%%%%%%%
\subsection{Band surgery}\label{sec:band_surgery}
\begin{definition}
Let $\Tcal=(T_1,\dots, T_n)$ be a multiplane diagram and let $(\rho,f)$ be an admissible tuple of framed arcs for $\Tcal$. The \emph{band surgery} of $\Tcal$ along $(\rho,f)$, denoted by $\Tcal[\rho,f]=(T_1[\rho_1,f_1], \dots,  T_n[\rho_n,f_n])$, is the new tuple where each tangle is obtained by band surgery as in Remark~\ref{rem:bands_as_arcs}. 
%can be surgered along an admissible tuple of framed arcs $(\rho,f)$ to yield a new tuple of tangles denoted by $\Tcal[\rho,f]=(T_1[\rho_1,f_1], \dots,  T_n[\rho_n,f_n])$. We say that $\Tcal[\rho,f]$ is the result of performing \emph{band surgery} to $\Tcal$ along $(\rho,f)$. We call this operation band surgery as bands (thus surgered tangles) are obtained from framed arcs as in Remark \ref{rem:bands_as_arcs}. 
\end{definition}

\begin{lemma}\label{lem:surgery_rhos}
If $(\rho,f)$ is an admissible tuple of framed arcs for a $b$--bridge multiplane diagram $\Tcal$, then $\Tcal[\rho,f]$ is the spine of a $(b+1)$--bridge multisected surface. 
\end{lemma}
\begin{proof}
We need to check that for each $i\in \Z_n$, $T[\rho_i,f_i]$ is a $(b+1)$-bridge trivial tangle and the surgered link $L_i[\rho_i\cup \rho_{i+1},f_i+f_{i+1}]$ is an unlink. % in $(b+1)$--bridge position for each pair of consecutive arcs $\rho_i$ and $\rho_{i+1}$. 
To see the first conclusion, recall that each $\rho_i$ is an increasing arc with one endpoint in $*\in \Sigma$; see Definition~\ref{def:admissible_arcs}\eqref{prop:increasing}. If $\rho_i$ is type 0, then $T_i[\rho_i,f_i]$ is equal to $T_i$ with an added 1--bridge strand around the band described by $(\rho_i,f_i)$. If $\rho_i$ is type 1, denote by $s$ the strand of $T_i$ containing one endpoint of $\rho_i$. We can slide the arc along $s$ until we reach the unique local maximum of the strand; this is possible as $T_i$ is a trivial tangle. The resulting arc $\rho_i$ can also be assume to satisfy Condition~\eqref{prop:increasing} in Definition~\ref{def:admissible_arcs}. In this position, it is easy to see that surgery along $(\rho_i,f)$ splits $s$ into two strands that have exactly one local maximum (near the maximum of $s$); note that the two strands will twist according to the framing $f_i$. Thus, $T_i[\rho_i,f_i]$ is trivial.
\\
We now explain why $L_i[\rho_i\cup \rho_{i+1},f_i+f_{i+1}]$ is an unlink. Lemma \ref{lem:framing_condition} ensures this whenever $\rho_i$ and $\rho_{i+1}$ are type 1. If both arcs are type zero, then $L_i[\rho_i\cup \rho_{i+1},f_i+f_{i+1}]$ is obtained from $L_i$ by adding a small 1--bridge unknot. If $\rho_i$ and $\rho_{i+1}$ are of different types, then $L_i[\rho_i\cup \rho_{i+1},f_i+f_{i+1}]$ is obtained by perturbing $L_i$. 
\end{proof}

Consider an admissible tuple $(\rho,f)$ and disks $\{\Delta_i\}_{i=1}^n$ as in Lemma \ref{lem:disk_from_arcs}. By construction, the surface represented by $\Tcal[\rho,f]$ differs from $F$ in a neighborhood of the disks $\Delta_1\cup \dots \cup \Delta_n$. The new disk system $\Dcal'_i$ is obtained from $\Dcal_i$ by modifying a neighborhood of $\Delta_i\cap \Dcal_i$ with certain surfaces near $\Delta_i$; if $\Delta_i\cap \Dcal_i$ is empty, we will add connected components to $\Dcal_i$. %with certain push-offs of $\Delta_i$. 
Figure \ref{fig:surgery_models} contains drawings of these modifications for the different types of $\rho_i$ and $\rho_{i+1}$, which we now describe. If both arcs $\rho_i$ and $\rho_{i+1}$ are type 0, then $\Delta_i$ is disjoint from $\Dcal_i$. In this case, $\Dcal'_i$ gains a small disk near $\Delta_i$ as in Part (a) of Figure \ref{fig:surgery_models}. If the types of $\rho_i$ and $\rho_{i+1}$ are distinct as in Part (b), the disk $\Dcal'_i$ is an elongated copy of $\Dcal_i$. When both $\rho_i$ and $\rho_{i+1}$ are type 1 and $\rho_i\cup \rho_{i+1}$ connects the same component of $L_i$, $\Dcal_i$ intersects a neighborhood of $\Delta_i$ in a rectangle as in Part (c). Here, if the framing $f_i+f_{i+1}$ is zero, then we replace this rectangle with two disjoint copies of $\Delta_i$. If the framing is $\pm 1$ instead, we add a half-twisted band connecting the two copies of $\Delta_i$ as in the right side of Part (c). As we see in Part (d), if $\rho_i+\rho_{i+1}$ connects distinct components of $L_i$, $\Dcal_i$ gains a band with twisting given by the framing of the arcs. 

%%%%%
\begin{figure}[h]
\centering
\includegraphics[width=.7\textwidth]{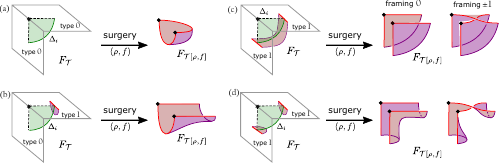}
\caption{Given an admissible tuple $(\rho,f)$, the disks $\Delta_i$ can be used to change $F_\Tcal$ in a controlled way inside each sector. %How the disks $\Dcal_i$ change after surgery along an admissible tuple $(\rho,f)$. 
We have four cases, depending on the types of the consecutive arcs $\rho_i$ and $\rho_{i+1}$ and whether $\rho_i\cup \rho_{i+1}$ connects (c) the same component of $L_i$ or (d) not. In case (c), we have two options, depending on the framing of $\rho_i\cup \rho_{i+1}$. In case (d), any choice of framing works; that said, we see that different framings produce twisting between the components of $\Dcal_i$. }
\label{fig:surgery_models}
\end{figure}
%%%%%

%%%%%%%%%%%%%%%%%%%%%%%%%%%%%%%%%%%%%%%%%%%%
%\subsubsection{Easy cases} 
%In the statement of Proposition \ref{thm:surgery_rhos_tubing}, we do not attempt to keep track of the number of tubings and crosscaps being attached to our multisected surface. %To give the reader a general idea on how to do this, 
Before we state the main technical result of this section (Thm \ref{thm:surgery_rhos_tubing}), we will describe examples of tuples $(\rho,f)$ for which we understand how $F_{\Tcal}$ and $F_{\Tcal[\rho,f]}$ are related. We delay explanations for Examples~\ref{ex:tubings}, \ref{ex:more_tubes}, and \ref{ex:crosscaps} until after Remark~\ref{rem:counting_tubes}, as this remark provides a concrete process for determining how $F_\Tcal$ changes.

\begin{example}[Tubings]\label{ex:tubings}
Fix $1\leq m \leq n-2$ and consider an admissible tuple $(\rho,f)$ having the following properties: 
(1) the arcs $\rho_i$ are type 1 for $i=1,\dots, m+1$ and type 0 otherwise, and
(2) for $i=1,\dots, m$, the framing $f_i+f_{i+1}$ is zero whenever $\rho_i\cup \rho_{i+1}$ connects the same component of $L_i$. 
If $r$ is the number of arcs $\rho_i\cup \rho_{i+1}$ connecting distinct components of $L_i$, then the surface $F_{\Tcal[\rho,f]}$ is obtained from $F_{\Tcal}$ by adding $r$ 1--handles. See Example~\ref{ex:explain_tubings} for an explanation.
\\
%This fact follows from Remark \ref{rem:counting_tubes}. %looking at Figure \ref{fig:surgery_proof_2}(a)--(d) and Figure \ref{fig:surgery_proof_1}. 
When $(n,m,r)=(n,m,0)$, $\Tcal[\rho,f]$ and $\Tcal$ represent isotopic surfaces. We will discuss more on this case in Section \ref{sec:perturbation}. 
When $(n,m,r)=(3,1,1)$ this corresponds to the tubing operation introduced in \cite{MTZ}.  Figure \ref{fig:exam_tube} shows the fact that the spin of a 2--bridge knot is obtained by adding one 1--handle to a 2--component unlink of 2--spheres.
\end{example}

%%%%%
\begin{figure}[h]
\centering
\includegraphics[width=.35\textwidth]{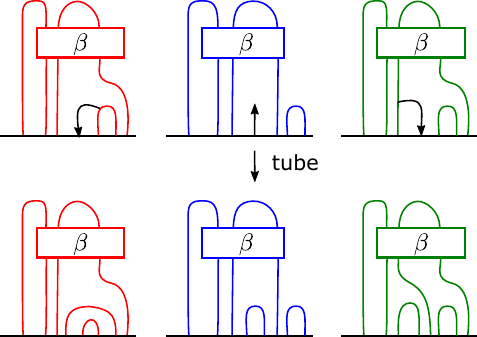}
\caption{Turning a $(3;2,2,3)$--bridge trisection of a 2--component unlink of 2--spheres into a $(4;2)$--bridge trisection of a spun knot. The top surface is unknotted by Proposition \ref{prop:completely_decomp}. }
\label{fig:exam_tube}
\end{figure}
%%%%%

\begin{example}[More Tubings]\label{ex:more_tubes}
Consider an admissible tuple $(\rho,f)$ satisfying the following condition:
(1) if $\rho_i$ and $\rho_j$ type 1 arcs for some $i\neq j$, then $2<|i-j|<n-1$. 
If $t$ is the number of type 1 arcs, then the surface $F_{\Tcal[\rho,f]}$ is obtained from $F_{\Tcal}$ by adding $(t-1)$ 1--handles. See Example~\ref{ex:explain_more} for an explanation.
\\
%This fact follows from Remark \ref{rem:counting_tubes}.  
%This fact can be checked by looking at Figure \ref{fig:surgery_proof_3}. 
When $t=1$, $\Tcal[\rho,f]$ and $\Tcal$ represent isotopic surfaces (see Remark \ref{rem:0_pert} for more on this case). %Section \ref{sec:perturbation}. 
%When $(n,m,r)=(3,1,1)$ this corresponds to the tubing operation introduced in \cite{MTZ}.  
Figure \ref{fig:exam_tube_2} shows how to use one 1--handle to turn a 2--component unlink of 2--spheres into the spin of a trefoil knot. We leave it as an exercise for the reader to see that this multiplane diagram is equivalent to some $\Tcal(\beta,(0,0))$ from Proposition \ref{prop:multisec_spun}. % like $\left(T_a(\beta), T_b(\beta,0), T_c(\beta), T_b(\beta,0)\right)$ as in Subsection \ref{sec:spun_knots}.
\end{example}

\begin{example}[Crosscap sums]\label{ex:crosscaps}
Fix $1\leq m \leq n-2$ and consider an admissible tuple $(\rho,f)$ having the following properties: 
(1) the arcs $\rho_i$ are type 1 for $i=1,\dots, m+1$ and type 0 otherwise, and
(2) for $i=1,\dots, m$, the arc $\rho_i\cup \rho_{i+1}$ connects the same component of $L_i$. 
If $r$ is the number of framings $f_i+f_{i+1}$ equal to $\pm 1$ for some $i=1,\dots, m$. Then the surface $F_{\Tcal[\rho,f]}$ is obtained from $F_{\Tcal}$ by taking the connected sum with $r$ unknotted projective planes. See Example~\ref{ex:explain_crosscap} for an explanation.
\\
In practice, one can use Lemma~\ref{lem:euler} to compare the Euler numbers of $F_{\Tcal}$ and $F_{\Tcal[\rho,f]}$; thus determining how many crosscap summands correspond to $\Pcal_+$ and $\Pcal_-$. More precisely, the non-negative integers $r_-$ and $r_+$ satisfying $F_{\Tcal[\rho,f]}=F_{\Tcal}\# r_-\Pcal_- \# r_+ \Pcal_+$ solve the equations
$$r=r_- + r_+, \text{ and } e\left(F_{\Tcal[\rho,f]}\right)=e\left(F_{\Tcal}\right)-2r_- + 2r_+.$$
%This fact follows from Remark \ref{rem:counting_tubes}. %This fact can be checked by looking at Figure \ref{fig:surgery_proof_2}(c)--(e). %When $(n,m,r)=(n,m,0)$, $\Tcal[\rho,f]$ and $\Tcal$ represent isotopic surfaces. We will discuss more on this case in Section \ref{sec:perturbation}. 
When $(n,m,r)=(3,1,1)$, this corresponds to the crosscap sum introduced in \cite{MTZ}. In Figure \ref{fig:exam_crosscap}, we show how to draw a 4-plane diagram for an unknotted Klein bottle by adding two projective planes to a 1--bridge diagram of a 2--sphere.  
\end{example}

% %%%%%
% \begin{figure}[h]
% \centering
% % \includegraphics[width=.5\textwidth]{images/fig_exam_crosscap.PNG}
% \includegraphics[width=.5\textwidth]{new_images/fig_10}
% \caption{Turning a $(1;1)$--bridge 4--section of an unknotted 2--sphere into a $(2;1)$--bridge 4--section of a Klein bottle.}
% \label{fig:exam_crosscap}
% \end{figure}
% %%%%%

% \begin{example}[$m$--sector perturbations]\label{ex:perturbation}
% Fix $0\leq m \leq n-2$ and consider an admissible tuple $(\rho,f)$ having the following properties: 
% (1) the arcs $\rho_i$ are type 1 for $i=1,\dots, m+1$ and type 0 otherwise, and
% (2) for $i=1,\dots, m$, $\rho_i\cup \rho_{i+1}$ connects the same component of $L_i$ with framing $f_i+f_{i+1}$ equal to zero.
% Then $\Tcal[\rho,f]$ and $\Tcal$ represent isotopic surfaces.
% \\
% This fact can be checked by looking at Figure \ref{fig:surgery_proof_2}(a)--(d). The isotopy between $F_{\Tcal[\rho,f]}$ and $F_{\Tcal}$ is supported in a neighborhood of \mbox{$X_1\cup \dots \cup X_m$} and consists in pushing \mbox{$\Dcal_1\cup \dots \cup \Dcal_m$} across $\Delta$ into \mbox{$X_{m+1}\cup \dots \cup X_n$.} We call this new multisection an \emph{$m$--sector perturbation of $\Tcal$}. Figures \ref{fig:quad_torus} and \ref{fig:quad_Klein} have examples of admissible arcs inducing two--sector perturbations. When $(n,m,r)=(3,1,1)$ this corresponds to the stabilization move for triplane diagrams introduced in \cite{MZ}. Section \ref{sec:perturbation} further explores this move.
% \end{example}

%%%%%%%%%%%%%%%%%%%%%%%%%%%%%%
The theorem below shows that most band surgeries change the underlying surface by 1--handle additions and crosscap sums. In practice, Remark \ref{rem:counting_tubes} gives a concrete process to decide how many tubes and crosscaps we add to $F_\Tcal$ to build $F_{\Tcal[\rho,f]}$.
\begin{theorem}\label{thm:surgery_rhos_tubing}
Let $\Tcal$ be a bridge multisection for $F\subset S^4$. Suppose that $(\rho,f)$ is an admissible tuple of framed arcs for $\Tcal$ with arcs of \textbf{both} types 0 and 1. Then, the surface represented by $\Tcal[\rho,f]$ can be obtained from $F$ by a sequence of isotopies, 1--handle additions, and crosscap sums. 
\end{theorem}

\begin{proof}
Note that if only one arc is type 1, then the surface represented by $\Tcal[\rho,f]$ is isotopic to $F$. In what follows, we will show how to reduce the number of type 1 arcs at the expense of removing a 1--handle or crosscap summand. Hence, $F_{\Tcal[\rho,f]}$ is obtained by adding 1--handles and crosscaps to an isotopic embedding of $F$.

First, we make sure that no consecutive\footnote{We consider the indices $n$ and $1$ to be consecutive as the tangles $T_n\cup \T_n$ bound disks in the multisected surface; see Definition~\ref{def:bridge_multisection}.} arcs are of type 1. Suppose, without loss of generality, that $\rho_1$, $\rho_2$, and $\rho_3$ are type 0, 1, and 1, respectively. Let $(\rho',f')$ be the tuple obtained by replacing $\rho_2$ from $\rho$ with a type 0 arc. We will see that $F_{\Tcal[\rho,f]}$ can be obtained from $F_{\Tcal[\rho',f']}$ by adding a 1--handle or (at most) one unknotted $RP^2$ summand. The first case we study is when $\rho_2\cup\rho_3$ connects distinct components of $L_2$. 
Near the disks $\Delta_1$ and $\Delta_2$, $F_{\Tcal[\rho,f]}$ looks like Part (d) of Figure \ref{fig:surgery_proof_1}. %Here, $F_{\Tcal[\rho,f]}$ looks like Part (d) of Figure \ref{fig:surgery_proof_1} near the disks $\Delta_1$ and $\Delta_2$. 
From Parts (a) and (b) of the same figure, we observe that $F_{\Tcal[\rho,f]}$ is obtained by adding one 1--handle to $F_{\Tcal[\rho',f']}$. 
Suppose now that $\rho_i\cup \rho_{i+1}$ connects the same component of $L_i$ as in Figure \ref{fig:surgery_proof_2}(c). Here, parts (d) and (e) depict the surgered surface obtained by choosing the framing $f_2+f_3$ to be 0 and $\pm 1$, respectively. One can see that changing the framing $f_2+f_3$ from zero to $\pm 1$ amounts to replacing a subdisk of $\Dcal_1\cup \Dcal_2$ with a M\"obius band. Thus, the surface $F_{\Tcal[\rho,f]}$ with $f_2+f_3=\pm 1$ is obtained by connect summing an unknotted $RP^2$ to the surface represented by $\Tcal[\rho,f]$ with $f_2+f_3=0$. In other words, we can assume that $f_2+f_3=0$. Now, from Parts (b) and (d) of Figure \ref{fig:surgery_proof_2}, we see that $F_{\Tcal[\rho,f]}$ is isotopic to $F_{\Tcal[\rho',f']}$ via an isotopy supported in a neighborhood of $\Delta_2$. 

From the process in the paragraph above, we are left with a tuple where no consecutive arcs are type 1. Without loss of generality, assume that $\rho_1$ is type 1 and $j$ is the smallest index $1<j<n$ with $\rho_j$ of type 1. Take $(\rho',f')$ to be the tuple obtained by replacing $\rho_j$ from $\rho$ with a type 0 arc. In Figure \ref{fig:surgery_proof_3}, we see that $F_{\Tcal[\rho,f]}$ is the result of adding a 1--handle to $F_{\Tcal[\rho',f']}$. %This finishes the proof. 
\end{proof}

%%%%%
\begin{figure}[h]
\centering
\includegraphics[width=.55\textwidth]{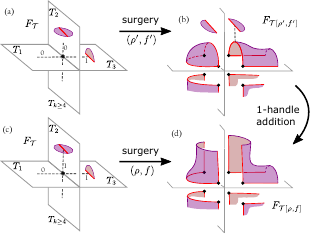}
\caption{Different choices of arcs yield surfaces related by 1--handle addition.}
\label{fig:surgery_proof_1}
\end{figure}
%%%%%

%%%%%
\begin{figure}[h]
\centering
\includegraphics[width=.7\textwidth]{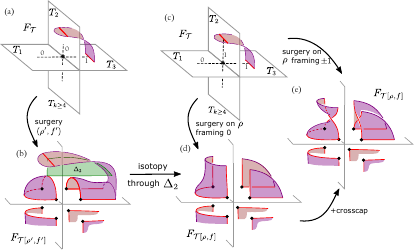}
\caption{Different choices of arcs yield surfaces related by at most one crosscap sum.}
\label{fig:surgery_proof_2}
\end{figure}
%%%%%

%%%%%
\begin{figure}[h]
\centering
\includegraphics[width=.55\textwidth]{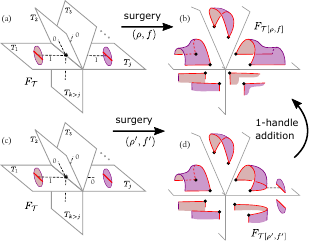}
\caption{Different choices of arcs yield surfaces related by 1--handle addition.}
\label{fig:surgery_proof_3}
\end{figure}
%%%%%

\begin{remark}[Counting 1--handles and crosscaps]\label{rem:counting_tubes}
Deciding how many 1--handles and crosscaps are added between $F_\Tcal$ and $F_{\Tcal[\rho,f]}$ is a combinatorial process. %This depends on the type of the arcs ($0$ or $1$) and the framings $f_i+f_{i+1}$ whenever $\rho_i\cup\rho_{i+1}$ connects the same component of $L_i$. For instance, s
Suppose we are given an admissible tuple $(\rho,f)$ with $\rho_1$ a type 0 arc. Then write a string in the characters $\{0,1,\overset{0}{=}, \overset{\pm}{=}, \neq\}$ by the following rules: 
(1) First write the types of the $\rho$--arcs in the order $\rho_1, \dots, \rho_n$. 
(2) Whenever two consecutive arcs are type 1, add a character between them. If $\rho_i\cup \rho_{i+1}$ connects the same components of $L_i$, add $\overset{0}{=}$ or $\overset{\pm}{=}$, depending if the framing $f_i+f_{i+1}$ is zero or $\pm 1$. If $\rho_{i}\cup \rho_{i+1}$ connects distinct components of $L_i$, add a $\neq$ symbol in between. 
For example, the string we obtain from the tuple in the 5-section diagram $\Tcal_3$ in Figure \ref{fig:k6_2} is \mbox{``$01\overset{0}{=} 1 \overset{\pm}{=} 1 \overset{0}{=} 1$''.} 

We now introduce the following rules that reduce the number of ones in our strings and correspond to specific 4-dimensional operations like 1--handle addition, crosscap sum, and isotopy. 
\begin{align}
\cdots 10\cdots 01\cdots &\overset{\text{1--h}}{\longleftarrow} \cdots 10\cdots 00\cdots \label{rule_1}
\\
% \]
% \[
\cdots 01\neq 1 \cdots &\overset{\text{1--h}}{\longleftarrow} \cdots 001\cdots \label{rule_2}
\\
% \]
% \[
\cdots 01\overset{\pm}{=} 1 \cdots &\overset{\# RP^2}{\longleftarrow} \cdots 001\cdots \label{rule_3}
\\
% \]
% \[
\cdots 01
\overset{0}{=}1 \cdots &\overset{\text{isotopy}}{\longleftarrow} \cdots 001\cdots \label{rule_4}
\end{align}

In Rule~\eqref{rule_1}, two non-consecutive type 1 arcs are separated by a string of some type zero arcs. The rule states that we can turn one of the type 1 arcs into a type 0 at the expense of removing a 1--handle from the multisected surface; this can be seen in Figure~\ref{fig:surgery_proof_3}. Rule~\eqref{rule_2} turns the substring $01\neq 1$ into $001$ at the expense of removing a 1--handle; this can be seen in Figure~\ref{fig:surgery_proof_1}. Rules~\eqref{rule_3} and \eqref{rule_4} explain how to replace the substring $01\overset{\varepsilon}{=}1$ with $001$ at the expense of removing a crosscap summand or without altering the isotopy class of the multisected surface; this is observed in Figure~\ref{fig:surgery_proof_2} depending on whether $\varepsilon =\pm 1$ or $0$. 

In practice, if the first tangle in our tuple $(\rho,f)$ is type 0, we can apply Rules~\eqref{rule_1}--\eqref{rule_4} to the corresponding string to get the word $100\cdots 0$. Here, one can see that for such a string, $\Tcal[\rho,f]$ and $\Tcal$ represent isotopic surfaces. 
%a word in $\{0,1\}$ with no consecutive ones. Here, by Example \ref{ex:more_tubes}, this final vector corresponds to $(t-1)$ 1--handles being added to $F_\Tcal$. 
\end{remark}

\begin{example}[Tubings]\label{ex:explain_tubings}
The admissible tuple in Example~\ref{ex:tubings} is described by the string $$1\neq 1 \neq \cdots \neq 1 0 \cdots 0$$ with $(m+1)$  consecutive 1's and $(n-m-1)$ 0's. Modulo cyclic permutation, this string can be reduced using Rule~\eqref{rule_2} to the string $100\cdots 0$. Thus, the band surgery described in Example~\ref{ex:tubings} changes the surface by $m$ 1--handle additions. 
\end{example}

\begin{example}[More tubings]\label{ex:explain_more}
The admissible tuple in Example~\ref{ex:more_tubes} is described by the string 
$$0\cdots 0 1 0 \cdots 0 1 0 \cdots 0 1 0 \cdots$$ 
where any pair of 1's is separated by at least one 0. Modulo cyclic permutation, this string can be reduced using Rule~\eqref{rule_1} to the string $100\cdots 0$. Thus, the band surgery described in Example~\ref{ex:more_tubes} changes the surface by as many 1--handle additions as one less than the number of type one arcs in $\rho$.
\end{example}

\begin{example}[Crosscap sums]\label{ex:explain_crosscap}
The admissible tuple in Example~\ref{ex:crosscaps} is described by the string 
$$1\overset{\pm}{=}1\overset{\pm}{=} \cdots \overset{\pm}{=}1 0 \cdots 0$$ with $(m+1)$ consecutive 1's and $(n-m-1)$ 0's. Modulo cyclic permutation, this string can be reduced using Rule~\eqref{rule_3} to the string $100\cdots 0$. Thus, the band surgery described in Example~\ref{ex:crosscaps} changes the surface by $m$ crosscap sums.
\end{example}

\begin{example}\label{fig:fig25}
We discuss how the admissible tuples in Figure \ref{fig:k6_2} change the multisected surfaces. 
From $\Tcal_3$ to $\Tcal_2$ we gain one crosscap as the sequence of reductions looks as follows, 
\[ 
01\overset{0}{=} 1 \overset{\pm}{=} 1 \overset{0}{=} 1 
\overset{\text{isotopy}}{\longleftarrow}
00 1 \overset{\pm}{=} 1 \overset{0}{=} 1 
\overset{\# RP^2}{\longleftarrow}
00 0 1 \overset{0}{=} 1 
\overset{\text{isotopy}}{\longleftarrow}
00 0 0 1 .
\]
From $\Tcal_2$ to $\Tcal_1$ we gain two 1--handles and one crosscap as the sequence of reductions looks as follows, 
\[ 
0 1 \neq 1 \overset{\pm}{=} 1 \neq 1 
\overset{\text{1-h}}{\longleftarrow}
0 0 1 \overset{\pm}{=} 1 \neq 1 
\overset{\# RP^2}{\longleftarrow}
0 0 0 1 \neq 1 
\overset{\text{1-h}}{\longleftarrow}
0 0 0 0 1.
\]
\end{example}

%%%%%%%%%%%%%%%%%%%%%%%%%%%%%%%
\section{Multiplane moves}\label{sec:multiplane_moves}
%{Moves that preserve the surface}
%{Applcation 2: Uniqueness of multiplane diagrams}
In this section, we show that any two multiplane diagrams describing the same embedded surface in $S^4$ are related by a finite sequence of multiplane moves. There are five types of moves: interior Reidemeister moves, mutual braid moves, multiple-sector perturbations, and split and merge moves. 

The first three moves are the multiplane versions of the triplane moves from \cite{MZ}. Fix a $b$--bridge multiplane diagram $\Tcal$. An \emph{interior Reidemeister move} is an isotopy of the tangles of $\Tcal$ that fixes the punctures. A \emph{mutual braid move} is the result of appending the same $2b$--stranded braid to each tangle of $\Tcal$. These correspond to ambient isotopies supported in a neighborhood of the central bridge surface. A \emph{multiple-sector perturbation} is the result of dragging the multisected surface $F$ through certain disk $\Delta$ embedded in multiple consecutive 4-dimensional sectors. This move is a particular band surgery of $\Tcal$ from Section \ref{sec:band_surgery} % along an admissible tuple as in Lemma \ref{lem:surgery_rhos}
%was introduced in Example \ref{ex:perturbation} 
and will be discussed in detail in Section \ref{sec:perturbation}. 

The last two types of multiplane moves, \emph{split} and \emph{merge}, change the number of sectors in a bridge multisection. They will be discussed in detail in Section \ref{sec:split_merge}. In Section \ref{sec:transfer}, we discuss a particular combination of the perturbations and merge moves that allow us to obtain a triplane diagram from any given multiplane diagram. This combination is key in proving our uniqueness statement for bridge multisections (Thm \ref{thm:main1}). 

\subsection{Multiple-sector perturbations} \label{sec:perturbation}
Let $\Tcal=(T_1,\dots, T_n)$ be a multiplane diagram for an embedded surface $F\subset S^4$. We say that a multiplane diagram $\Tcal'$ is a \emph{multiple-sector perturbation} of $\Tcal$ if $\Tcal'=\Tcal[\rho,f]$ for an admissible tuple of framed arcs $(\rho,f)$ satisfying the following conditions for some integer $0\leq m\leq n-2$
\begin{enumerate}
\item after a cyclic permutation of the indices, the type 1 arcs are $\rho_i$ for $i=1,\dots, m+1$, and % and type 0 otherwise, and
\item for $i=1,\dots, m$, $\rho_i\cup \rho_{i+1}$ connects the same component of $L_i$ with framing $f_i+f_{i+1}$ equal to zero.
\end{enumerate}
We may also refer to $\Tcal'$ as an \emph{$m$--sector perturbation} of $\Tcal$. Observe that $m$--sector perturbation corresponds to the band surgery in Example~\ref{ex:tubings} with $(m,r)=(m,0)$.

\begin{lemma}\label{lem:perturbations}
Let $\Tcal$ be a $(b;c_1,\dots, c_n)$--multiplane diagram for $F$. 
If $\Tcal'=\Tcal[\rho,f]$ is a multiple--sector perturbation of $\Tcal$, then $\Tcal'$ has underlying surface isotopic to $F$. Moreover, if the type 1 arcs of $\rho$ are $\rho_1,\dots, \rho_{m+1}$, then $\Tcal'$ has invariants 
\[(b;c_1, \dots, c_m, \dots, c_{n-1}, c_n)+\left(1;1,\dots, 1,\stackrel{-(m+1)-}{0},1,\dots, 1,\stackrel{-n-}{0}\right).\]
\end{lemma}

\begin{proof}
After a cyclic permutation, we can assume that the first $m+1$ arcs of $\rho$ are type 1 and the rest are type zero. By Theorem \ref{thm:surgery_rhos_tubing}, we know that $\Tcal'$ is a multiplane diagram representing a surface obtained by adding 1--handles and crosscaps to $F$. We claim that no 1--handles nor crosscaps are needed. As $(\rho,f)$ is an admissible tuple, we can consider the disks $\Delta_1, \dots, \Delta_n$ as in Lemma \ref{lem:disk_from_arcs}. There is an isotopy between $F_{\Tcal[\rho,f]}$ and $F_{\Tcal}$ supported in a neighborhood of \mbox{$X_1\cup \dots \cup X_m$}. This consists in pushing \mbox{$\Dcal_1\cup \dots \cup \Dcal_m$} across $\Delta_1\cup \cdots \cup \Delta_m$ into \mbox{$X_{m+1}\cup \dots \cup X_n$} (see Figure \ref{fig:surgery_proof_2}(a)--(d)). 

The change of the patch numbers depends only on the types of the arcs $\rho_i$ and $\rho_{i+1}$. If $1\leq i \leq m$, 0-framed surgery along $\rho_i\cup \rho_{i+1}$ divides one component of $L_i$ into two, thus $c'_i=c_i + 1$. If $m+2\leq i \leq n-1$, then $\rho_i\cup \rho_{i+1}$ is disjoint from $L_i$ and so $L'_i$ gains a small 1--bridge unknotted component. Thus $c'_i=c_i+1$. If $i=m+1, n$, then $\rho_i\cup \rho_{i+1}$ intersects $L_i$ in one endpoint and so $L'_i$ is obtained by pushing $L_i$ through the bridge sphere guided by $\rho_i\cup \rho_{i+1}$. Thus, $L'_i$ is isotopic to $L_i$ and $c'_i=c_i$. 
\end{proof}

In Lemma \ref{lem:perturbations}, we push our surface $F$ along a disk in the first $m$ sectors, increasing the bridge and patch numbers ($b$ and $c_i$'s) by one, except for $c_{m+1}$ and $c_n$. If $n=3$ and $m=1$, i.e. a classic trisection perturbation, this process yields a $(b+1;c_1+1,c_2,c_3)$--bridge trisection \cite[Lem 6.1]{MZ}. Figures \ref{fig:quad_torus} and \ref{fig:quad_Klein} have examples of admissible arcs inducing two-sector perturbations. 

\begin{remark}[Cases $m\geq n-1$] Notice that $m$--sector perturbations are not defined for $m=n-1,n$. One can interpret an $m$--perturbation as transferring saddles of $F$ from the sectors $\Dcal_1,\dots, \Dcal_m$ to the rest of the sectors. Intuitively, if $m=n-1$, then the resulting surface $F'\cap X_n$ may not be a trivial disk system. If $m=n$, `perturbing' along a disk $\Delta\subset X_1\cup \dots \cup X_n$ can be interpreted as a compression of $F$ along $\Delta$, thus changing the knotted surface. 
\end{remark}

\begin{remark}[Case $m=0$]\label{rem:0_pert}
%0--Sector perturbations occur when exactly one arc of $\rho$ is type 1. 
0--Sector perturbations change the multisected surface by an isotopy supported in a neighborhood of one 3--ball $B_i$. They are determined by one type 1 arc regardless of the framing. Following the definition of band surgery, one can see that the set of non-trivial links appearing as a subset of a cross-section of $\Tcal$ is not affected by 0--sector perturbations. Although they are defined, 0--sector perturbations are redundant: if $\Tcal'$ is obtained from $\Tcal$ by a 0-sector perturbation then $\Tcal'$ is also obtained from $\Tcal$ by a \mbox{$(n-2)$--sector} perturbation. %In Section \ref{sec:connecting_multiplane_diagrams} we use them in our arguments. %exploit the fact that we only need to specify one type 1 arc to perform a 0--sector perturbation.
\end{remark}

%%%%%%%%%%%%%%%%%%%%%%%%%%
\subsection{Split and Merge}\label{sec:split_merge}
Let $\Tcal=(T_1,\dots, T_n)$ be a multiplane diagram and let $\Tcal'$ be the tuple obtained from $\Tcal$ by removing one tangle from $\Tcal$. We say that $\Tcal$ is obtained from $\Tcal'$ by \emph{merging a sector} if $\Tcal'$ is a multiplane diagram with $F_{\Tcal'}$ isotopic to $F_{\Tcal}$. It is important to note that removing an arbitrary tangle from a multiplane diagram may not result in a new multiplane diagram. % (see Definition [multiplane diagram]). 
Thus, not all multiplane diagrams admit a merge move. For instance, triplane diagrams of non-trivial 2--knots cannot be merged. The following is a criterion for when a merge move exists. 

\begin{lemma}\label{lem:merging}
Let $\Tcal=(T_1,T_2,T_3,\dots, T_n)$ be a multiplane diagram with $n\geq 3$. Suppose that $\Tcal_0=(T_1,T_2,T_3)$ is a $(b;c_1,c_2,c_3)$--triplane diagram for a $c_3$--component link of 2--spheres. If $\Tcal_0$ is completely decomposable, then $\Tcal'=(T_1,T_3,T_4,\dots, T_n)$ is a multiplane diagram with underlying surface isotopic to $F_\Tcal$. 
\end{lemma}
\begin{proof}
As $\Tcal_0$ is a triplane diagram, we know that $T_1\cup \T_3$ is a $c_3$--component unlink. Hence, the tuple $\Tcal'$ is a multiplane diagram. It remains to show that $F_{\Tcal'}$ is isotopic to $F_{\Tcal}$. Denote by $F_0$ the surface represented by $\Tcal_0$ and by $(S^4,F_0)=\cup_{i=1}^3 (X_i, \Dcal_i)$ 4-dimensional sectors of $F_0$. Note that $\Dcal_1$ and $\Dcal_2$ are also consecutive disk systems for the multisected surface $F_\Tcal$. The action of removing $T_2$ from $\Tcal$ corresponds to removing $(X_1,\Dcal_1)\cup_{T_2}(X_2,\Dcal_2)$ from $F_\Tcal$ and replacing it with $(X_3,\Dcal_3)$. As $\Tcal_0$ is completely decomposable, both $\Dcal_1\cup \Dcal_2$ and $\Dcal_3$ are isotopic (in the 4-ball) to spanning surfaces for the link $T_1\cup \T_3$. As we know $\Dcal_3$ is a collection of $c_3$ disks and $F_{\Tcal_0}$ has $c_3$ connected components, $\Dcal_1\cup \Dcal_2$ must be a trivial system of $c_3$ disks bounded by $T_1\cup \T_3$. In particular, the disks $\Dcal_1\cup \Dcal_2$ and $\Dcal_3$ are isotopic relative to the link $T_1\cup \T_3$. Hence, $F_{\Tcal'}$ and $F_{\Tcal}$ are isotopic. 
\end{proof}

We now translate the merging criterion in Lemma \ref{lem:merging} for merging a multisection in terms of dual bands. This may be used in practice for multiplane diagrams with many crossings. Recall that, by Lemma \ref{lem:dual_bands}, consecutive tangles in a multiplane diagram are related by surgery along dual bands. In particular, if three consecutive tangles of $\Tcal$ satisfy $T_1[v_1]=T_2$ and $T_2[v_2]=T_3$, then we can draw the bands $v_2$ in the same 3--ball as $T_1$ and $v_1$. If $v_2$ and $v_1$ share some endpoints, we choose to draw $v_2$ closer to the bridge surface than $v_1$ as in the top left of Figure \ref{fig:tansfer_bands}. We drew concrete examples of such bands in Figure \ref{fig:tansfer_spun}. 

\begin{lemma}\label{lem:merging_bands}
Let $\Tcal=(T_1,T_2,T_3,\dots, T_n)$ be a multiplane diagram with $n\geq 3$. Let $v_1$ and $v_2$ be bands dual to $T_1$ and $T_2$, respectively, satisfying, $T_2=T_1[v_1]$ and $T_3=T_2[v_2]$. If $v=v_1\cup v_2$ is a dual set of bands for $T_1$, then $\Tcal'=(T_1,T_3,T_4,\dots, T_n)$ is a multiplane diagram with underlying surface isotopic to $F_\Tcal$.
\end{lemma}

\begin{proof}
The fact that $v=v_1\cup v_2$ is dual to $T_1$ implies that $T_1\cup \T_3$ is a $(b-|v|)$--component unlink and so the triplet $\Tcal_0=(T_1,T_2,T_3)$ is a $(b;b-|v_1|,b-|v_2|,b-|v|)$--triplane diagram. Moreover, after braid and interior Reidemeister moves, the tangle and the bands $T_1\cup v$ can be crossingless as in Figure \ref{fig:dual_bands}. 
Thus, $T_2=T_1[v_1]$ and $T_3=T_1[v]$ are also crossingless tangles and so $\Tcal_0$ is completely decomposable with underlying surface $F_{\Tcal_0}$. As the Euler characteristic of $F_{\Tcal_0}$ is $2(b-|v|)$, $F_{\Tcal_0}$ must be a $(b-|v|)$--component unlink of unknotted 2--spheres. Hence, by Lemma \ref{lem:merging}, $\Tcal'$ is a multiplane diagram with $F_{\Tcal'}$ isotopic to $F_{\Tcal}$.
\end{proof}

The \emph{split move} is the opposite of the merge move. More precisely, let $\Tcal$ be a multiplane diagram and let $\Tcal'$ be the tuple obtained from $\Tcal$ by \emph{inserting} one tangle into $\Tcal$. We say that $\Tcal$ is obtained from $\Tcal'$ by \emph{splitting a sector} if $\Tcal'$ is a multiplane diagram with $F_{\Tcal'}$ isotopic to $F_{\Tcal}$. Lemmas \ref{lem:merging} and \ref{lem:merging_bands} can also be used as criteria for split moves. For example, Theorem \ref{thm:thin_multi} is a consequence of Lemma \ref{lem:merging_bands}. 

\begin{theorem}\label{thm:thin_multi}
Every surface $F$ in $S^4$ admits a thin bridge multisection. 
\end{theorem}

\begin{proof}
Let $\Tcal=(T_1,T_2,T_3)$ be a bridge trisection of $F$; this exists by \cite{MZ}. We will split each sector of $\Tcal$ as follows. By Lemma \ref{lem:dual_bands}, we can find a set $v=\{v_1,\dots, v_{m_1}\}$ of bands dual to $T_1$ satisfying $T_2=T_1[v]$. Define the tangles $T^{(0)}_1=T_1$ and $T^{(i+1)}_1=T^{(i)}_1[v_i]$ for $i=1,\dots m_1$. We do the same for the pairs of tangles $(T_2,T_3)$ and $(T_3,T_1)$ and define tangles $T^{(j)}_2$ and $T^{(l)}_3$ accordingly. We obtain a multiplane diagram $$\Tcal'=\left(T^{(0)}_1,\dots, T^{(m_1-1)}_1, T^{(0)}_2,\dots, T^{(m_2-1)}_2, T^{(0)}_3,\dots, T^{(m_3-1)}_3\right)$$ where consecutive tangles differ by one dual band, hence $\Tcal'$ is thin. After several applications of Lemma \ref{lem:merging_bands}, we can conclude that $F$ is the underlying surface of  $\Tcal'$.
\end{proof}

\begin{remark}[Motivation]
Thin multisections permit the study of knotted surfaces as loops in certain complexes. Some of Engelhardt's PhD dissertation discusses how loops in the pants complex of planar surfaces codify the same information as a thin bridge multisection. For closed 4-manifolds, the study of thin multisections was carried out by Islambouli in \cite{UPW}.
\end{remark}

In light of Theorem~\ref{thm:thin_multi}, one can study bounds on the number of sectors needed to achieve a thin bridge multisection of a given knotted surface. If $n=3$, it follows from Proposition~\ref{prop:completely_decomp_2} that thin bridge trisections represent unlinks of unknotted 2-spheres and at most one unknotted projective plane. The unknotted torus admits a thin bridge 4--section; see Figure~\ref{fig:quad_torus}. 

\begin{problem}
%Let $F\subset S^4$ be a knotted surface. 
Compute the smallest $n\geq 3$ such that $F\subset S^4$ admits a thin bridge $n$--section.
\end{problem}

%%%%%%%%%%%%%%%%%%%%%%%%%%%%%%%%%%%%%%
\subsection{Perturbation-merging combination} \label{sec:transfer}
Proposition \ref{prop:transfer_lem} combines multiple-sector perturbations and merging moves to reduce the number of sectors of any multiplane diagram with $n\geq 4$ sectors. In particular, one can always recover a triplane diagram from a given multiplane diagram. Figures \ref{fig:quad_torus} and \ref{fig:quad_Klein} are concrete examples of this process.

%%%%%
\begin{figure}[h]
\centering
\includegraphics[width=.45\textwidth]{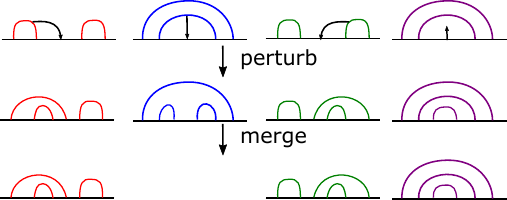}
\caption{Turning a $(2;1)$--bridge 4--section into a $(3;1)$--trisection of an unknotted torus.}
\label{fig:quad_torus}
\end{figure}
%%%%%
%%%%%
\begin{figure}[h]
\centering
\includegraphics[width=.45\textwidth]{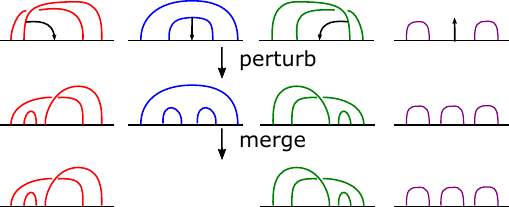}
\caption{Turning a $(2;1)$--bridge 4--section into a $(3;1)$--trisection of a Klein bottle.}
\label{fig:quad_Klein}
\end{figure}
%%%%%

\begin{proposition}\label{prop:transfer_lem}%[Transfer move]
Let $n\geq 4$ and let $\Tcal=(T_1,T_2,T_3,\dots, T_n)$ be a $(b;c_1,c_2,\dots, c_n)$--multiplane diagram. There exist $(b-c_2)$ two-sector perturbations such that the resulting multiplane diagram $(T'_1, T'_2, T'_3, \dots, T'_n)$ can be merged into a multisection with $(n-1)$ sectors $(T'_1, T'_3, T'_4, \dots, T'_n)$ with invariants equal to
   \[ \left(2b-c_2;c_1,c_3, c_4+(b-c_2), \dots, c_{n-1}+(b-c_2), c_n\right)\] 
\end{proposition}
In the statement above, we merge the sectors $\Dcal'_1$ and $\Dcal'_2$. In particular, the entry with $c_2$ disappears, and all the invariants ($b$ and $c_i$'s) increase by $(b-c_2)$ except for $c_1$, $c_3$ and $c_n$. If $n=4$, this process would yield a $(2b-c_2;c_1,c_3,c_4)$--bridge trisection. 
\begin{proof}
The figures referenced in this proof depict the case $c_2=|T_2\cup \T_3|=1$; that being said, all of our arguments carry over to the case $c_2\geq 2$. After mutual braid moves and interior Reidemeister moves, we can represent $T_2$ and $T_3$ by a crossingless diagram as in Figure \ref{fig:transfer_1}. In this figure, the arcs (with blackboard framing) form admissible tuples for $\Tcal$. Thus, they induce a sequence of $(b-c_2)$ two-sector perturbations along disks in $X_1\cup X_2$. We denote by $\Tcal'=(T_1', \dots, T_n')$ the resulting bridge multisection. By Lemma \ref{lem:perturbations}, the invariants of $\Tcal'$ are given by the new tuple
\[(b;c_1, c_2, c_3, c_4, c_5, \dots, c_{n-1}, c_n)+(b-c_2)\cdot\left(1;1,1,\stackrel{-3-}{0},1,1,\dots, 1,\stackrel{-n-}{0}\right).\]

%%%%%
\begin{figure}[h]
\centering
\includegraphics[width=.7\textwidth]{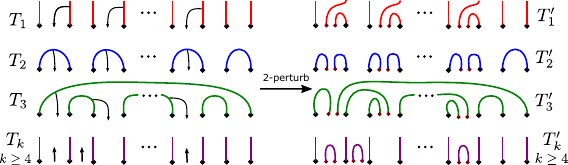}
\caption{Two--sector perturbations.}
\label{fig:transfer_1}
\end{figure}
%%%%%

We observe in Figure \ref{fig:transfer_2} that the link $T'_1\cup \T'_3$ is isotopic to $T_1\cup \T_2$. Thus, \mbox{$\Tcal_0=(T'_1, T'_2, T'_3)$} is a \mbox{$(2b-c_2;c_1+b-c_2, b, c_1)$--}triplane diagram. 
%$(b+(b-c_2);c_1+ (b-c_2), b, c_1)$--bridge trisection.
Via an Euler characteristic computation, we see that the underlying surface for $\Tcal_0$ is a $c_1$--component link of 2--spheres. To apply Lemma \ref{lem:merging}, we must verify that $\Tcal_0$ is completely decomposable. In Figure \ref{fig:transfer_3} we show a sequence of $(b-c_2)$ 1-sector deperturbations that turn $\Tcal_0$ into a $(b;c_1,b,c_1)$--bridge trisection. By Proposition \ref{prop:completely_decomp}, the resulting trisection is completely decomposable and so is $\Tcal_0$. Lemma \ref{lem:merging} implies then that $(T'_1, T'_3, T'_4, \dots, T'_n)$ is a multiplane diagram with underlying surface $F_{\Tcal}$, as desired. 
    %%%%%
\begin{figure}[h]
\centering
\includegraphics[width=.6\textwidth]{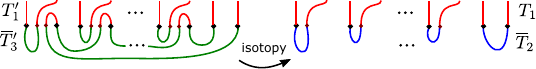}
\caption{The links $T'_1\cup \T'_3$ and $T_1\cup \T_2$ are isotopic.}
\label{fig:transfer_2}
\end{figure}
%%%%%
%%%%%
\begin{figure}[h]
\centering
\includegraphics[width=.95\textwidth]{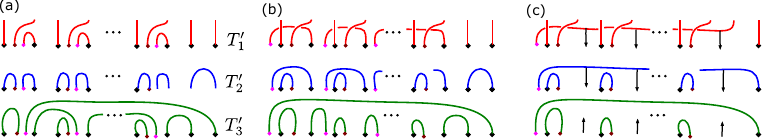}
\caption{From (a) to (b), one slides the purple punctures to the left. The resulting triplane diagram admits $(b-c_2)$ deperturbations turning (b) into (c).}
\label{fig:transfer_3}
\end{figure}
%%%%%
%
\end{proof} 

%%%%%%%%%%%%%%%%%%%%%%%%%%%%%%%%%%%%%%%%%%%%%
%\subsection{Transfer move with bands}$\quad$

We now explain the perturbation-merging process in Proposition \ref{prop:transfer_lem} using bands instead of standard pairs. This may be useful in practice as dual bands can guide the necessary 2--sector perturbations. Figure \ref{fig:tansfer_spun} is an example of such a procedure. In there, we turn the $(3;2)$--bridge 4--sections of $m$--twist spun 2--bridge knots from Proposition \ref{prop:multisec_spun} into the $(4;2)$--bridge trisections described by Meier and Zupan in \cite{MZ}.

%%%%%
\begin{figure}[h]
\centering
\includegraphics[width=.5\textwidth]{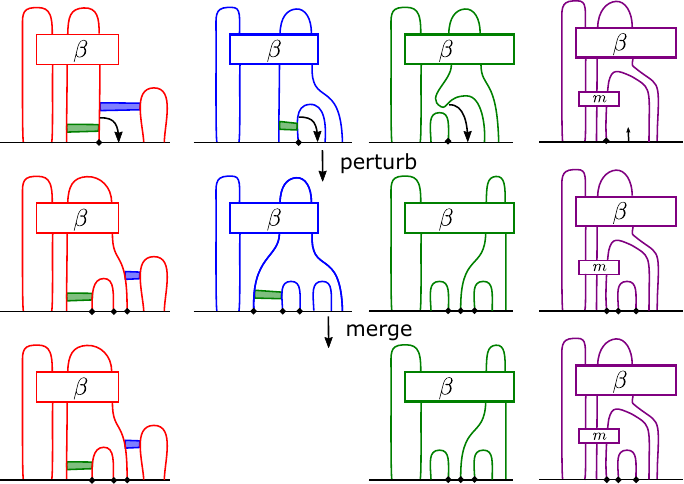}
\caption{Turning a $(3;2)$--bridge 4--section into a $(4;2)$--trisection of an $m$--twist spun knot.}
\label{fig:tansfer_spun}
\end{figure}
%%%%%

Let $\Tcal=(T_1,\dots, T_N)$ be a multiplane diagram with $N\geq 4$. Let $v_1$ and $v_2$ be bands dual to $T_1$ and $T_2$, respectively, satisfying $T_2=T_1[v_1]$ and $T_3=T_2[v_2]$. 
By the duality condition, the cores of the bands $v_1$ and $v_2$ are embedded arcs in the bridge surface $\Sigma$. So we can draw $v_2$ as bands for the tangle $T_1$ with cores lying in a copy of $\Sigma$ ``lower'' than the one where the bands of $v_1$ lie; see top left of Figure \ref{fig:tansfer_bands}. 
Then, for each band of $v_2$, we pick one puncture of $\Sigma$ adjacent to $v_2$ and modify the tangles in $\Tcal$ as in Figure \ref{fig:tansfer_bands}. One can check that the bands $v=v_1\cup v_2$ are dual to the new tangle $T'_1$. Thus, the new tuple $\left(T'_1,T'_2=T'_1[v_1],T'_3=T'_2[v_2]\right)$ satisfies the condition of Lemma \ref{lem:merging_bands}. After merging the first two sectors, we obtain a multiplane diagram $\Tcal'=(T'_1,T'_3,T'_4,\dots, T'_N)$ with $F_{\Tcal'}=F_{\Tcal}$.
%%%%%
\begin{figure}[h]
\centering
\includegraphics[width=.35\textwidth]{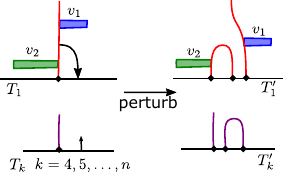}
\caption{Local modification around the endpoint of a band in $v_2$.}
\label{fig:tansfer_bands}
\end{figure}
%%%%%

%%%%%%%%%%%%%%%%%%%%%%%%%%%%%%%%%%%%%%%%%%%%%%%%%%%%%%%%%
We are ready to prove the main theorem of this section. 

\begin{theorem}\label{thm:main1}
Let $\Tcal_1$ and $\Tcal_2$ be two multiplane diagrams representing isotopic surfaces in $S^4$. There is a finite sequence of multiplane moves that turns $\Tcal_1$ into $\Tcal_2$. 
% Any two multiplane diagrams of a given knotted surface are related by a finite sequence of multiplane moves. 
\end{theorem}
\begin{proof}
Let $\Tcal_1$ and $\Tcal_2$ be two multiplane diagrams with isotopic underlying surfaces. Proposition \ref{prop:transfer_lem} applied to each $\Tcal_i$ gives a sequence of two-sector perturbations and merging moves turning $\Tcal_i$ into a triplane diagram $\Tcal'_i$. To end, by \cite[Thm 1.7]{MZ}, there is a finite sequence of triplane moves (thus multiplane moves) relating $\Tcal'_1$ with $\Tcal'_2$. 
\end{proof}

% {\color{red} Rant about the necessity of the merging/splitting move.} Work of J. Williams [CITE] suggests that one could go from two multisections with the number of sectors by a sequence of multiplane moves that do not change the number of 4-dimensional sectors. That said, proof of such a claim is needed.

% {\color{red} wonder if} all of the multiplane moves are necessary. Split/merging? ... What about multiple-sector perturbations? 

%%%%%%%%%%%%%%%%%%%%%%%%%%%%%%%%%%%%%%%%%%%
\section{Band surgeries are unknotting moves}
%{Application 3: Connecting multiplane diagrams}
\label{sec:connecting_multiplane_diagrams}

This section shows how to connect two multiplane diagrams for possibly distinct surfaces. The technical bit is Proposition \ref{prop:multisecting_tubes}, where we show that any 1--handle addition can be achieved via band surgeries on multisection diagrams. 

\begin{proposition}\label{prop:multisecting_tubes}
Let $\Tcal$ be a multiplane diagram representing $F\subset S^4$. Assume that $F'$ is obtained by 1--handle addition on $F$. Then there is a multiplane diagram $\Tcal'$ and admissible arcs $(\rho,f)$ for $\Tcal'$ such that 
\begin{itemize}
    \item $\Tcal'$ is obtained from $\Tcal$ by a sequence of 0--sector perturbations, and 
    \item $\Tcal'[\rho,f]$ is a multiplane diagram for $F'$.
\end{itemize}
\end{proposition}

\begin{proof}
Let $t$ be the guiding arc of the 1--handle addition between $F$ and $F'$. By definition, $t$ is an embedded arc in $S^4$ intersecting $F$ only in its endpoints. As the spine of $\Tcal$ is the 1--skeleton of $F$, we can slide the endpoints of $t$ to lie in distinct tangles. The first step of the proof is to isotope $t$ to intersect each tangle in a particular way--see conditions (a)--(c) below. 

Let $N(\Sigma)$ be a small tubular neighborhood of $\Sigma$ in $S^4$. We choose $N(\Sigma)$ so that $N(\Sigma)\cap F$ is the disjoint union of $2b$ disks around the bridge points. For each $i=1,\dots, n$, denote by $N_i(\Sigma)=N(\Sigma)\cap B_i$ the a small neighborhood of $\Sigma$ in $B_i$ with boundary $\Sigma \cup \Sigma_i$. Note that $\Sigma_i$ intersects $T_i$ in $2b$ points and the subset of $T_i$ inside $B_i\setminus int(N(\Sigma))$ is still a $b$--bridge trivial tangle. By construction, the spheres $\Sigma_1, \dots, \Sigma_n$ are a subset of $\partial N(\Sigma)\approx \Sigma\times S^1$. 

%%%%%
\begin{figure}[ht]
\centering
\includegraphics[width=.6\textwidth]{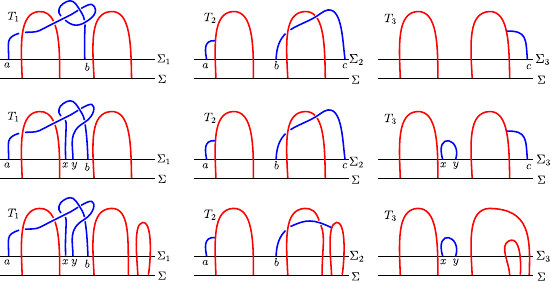}
\caption{Top row is a guiding arc $t$ that intersects a triplane diagram with conditions (a) and (b). From top to middle, we dragged the minimum of the blue arc in $T_1$ towards $\Sigma_1$. Thus, getting red/blue tangles with only maxima satisfying (c). From middle to bottom, we perform a 0--sector perturbation using as the type 1 arc the blue arc connecting the point $c$ with the tangle $T_3$.}
\label{fig:multisecting_tubes_1}
\end{figure}
%%%%%

By transversality, $t$ can be isotoped to be disjoint from the central surface $\Sigma$; thus $t\cap int(N(\Sigma))=\emptyset$. As arcs in each 4-ball $X_i$ can be pushed towards $\partial X_i$, we can isotope $t$ to satisfy the following conditions:
% \begin{itemize}
% \item[(a)] 
\textbf{(a)} $t$ intersects each $B_i$ in arcs embedded in $B_i\setminus int(N_i(\Sigma))$ with endpoints in $\Sigma_i$ or $T_i$, and 
% \item [(b)] 
\textbf{(b)} the endpoints of $t$ in $\Sigma_i$ are connected between each other by arcs in $\partial N(\Sigma)$ of the form $\{pt\}\times I$ for some interval $I$ of $S^1$. 
% \end{itemize}
We depict an example of this situation at the top of Figure \ref{fig:multisecting_tubes_1}. In the figure, we labeled the points in $t\cap \Sigma_i$ with the same symbol whenever product arcs in $\partial N(\Sigma)$ connect them. 

We can further isotope the arcs in $t\cap B_i$, at the expense of increasing the intersections $t\cap \Sigma_i$, so that \textbf{(c)} $t\cap B_i$ is a collection of boundary parallel arcs in $B_i-int(N(\Sigma))$ and at most two increasing arcs as in the second row of Figure \ref{fig:multisecting_tubes_1}. This can be achieved by making the arcs $t\cap B_i$ Morse and pushing all their minima towards $\Sigma_i$. Each minimum will split an arc of $t\cap T_i$ into two arcs with fewer critical points, creating a pair of points in $t\cap \Sigma_i$, and one arc with one maximum point in the tangle $t\cap B_{i\pm 1}$--see new points $x$ and $y$ in Figure \ref{fig:multisecting_tubes_1}.

The second step in the proof is to use the arcs in $t\cap \left(\cup_{i=1}^n B_i\right)$ to guide a sequence of 0--sector perturbations of $\Tcal$. As mentioned in Remark \ref{rem:0_pert}, 0--perturbations are determined by one type 1 arc with arbitrary framing. The first type 1 arc we use is one from $t\cap \left(\cup_{i=1}^n B_i\right)$ that intersects $\Sigma_i$ and $T_i$ once. After the 0--perturbation, the guiding arc $t$ will intersect $\cup_{i=1}^n B_i$ in one less arc. The third row of Figure \ref{fig:multisecting_tubes_1} exemplifies how one can begin this process. We continue the sequence of 0--perturbations until we are left with a guiding arc, denoted again by $t$, that intersects the interior of exactly two distinct 3--balls, say $B_1$ and $B_j$ for some $j\neq 1$. The arc $t\cap B_1$ (resp. $t\cap B_j$) has no critical points and connects one point in $\Sigma$ 
%$\Sigma_1$ (resp. $\Sigma_j$) 
with one point of $T_1$ (resp. $T_j$). 
We have two options, depending on whether $j=2,n$ or not. 
If $j\neq 2,n$, we can consider a tuple of arcs $\rho$ with exactly two type 1 arcs: $\rho_1= t\cap B_1$ and $\rho_j= t\cap B_j$. As $\rho_1$ and $\rho_j$ are not consecutive arcs, $(\rho,f)$ is an admissible tuple regardless of the framing $f$. Moreover, by Example \ref{ex:more_tubes}, banded surgery along $(\rho,f)$ represents a 1--handle addition to $F$ with guiding arcs $\rho_1\cup \rho_j=t$. 

To end, suppose that $j=2$ (the same argument works for $j=n$). The tuple of arcs $\rho$ defined in the previous paragraph may not be admissible, or $\rho_1\cup\rho_2$ may connect the same component of $T_1\cup \T_2$. To ensure the admissibility condition, we perform additional multiple-sector perturbations described in Figure \ref{fig:multisecting_tubes_2}. First, we use the increasing arcs $t\cap B_1$ and $t\cap B_2$ to guide two consecutive 0--sector perturbations. We follow this with consecutive 1--sector perturbations as in Figure \ref{fig:multisecting_tubes_2}. The link $T_1\cup \T_2$ for the resulting multiplane diagram has two new 1--bridge unknots (obtained from the 1--sector perturbations). Moreover, the guiding arc $t$ can be slid to connect those two new unknots as in Figure \ref{fig:multisecting_tubes_2}. In particular, we obtain an admissible tuple of $\rho$ with $\rho_1=t\cap B_1$, $\rho_2=t\cap B_2$, and $\rho_i$ type 0 for $i\neq 1,2$. Since $\rho_1$ and $\rho_2$ now connect different components of the link $T_1\cup \T_2$, any choice of framing for $\rho_1$ and $\rho_2$ will make the tuple $(\rho,f)$ admissible. %By Lemma \ref{lem:framing_condition}, $(\rho,f)$ is an admissible tuple for any framing $f$. 
By Example \ref{ex:tubings}, banded surgery along $(\rho,f)$ represents a 1--handle addition to $F$ with guiding arcs $\rho_1\cup \rho_2=t$. This finishes the proof of the proposition.
\end{proof}
%%%%%
\begin{figure}[h]
\centering
\includegraphics[width=.7\textwidth]{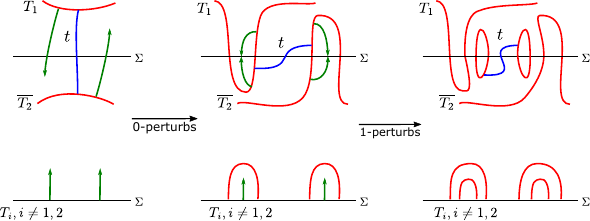}
\caption{Sequence of perturbations to make $t=\rho_1\cup \rho_2$ an admissible pair of arcs.}
\label{fig:multisecting_tubes_2}
\end{figure}
%%%%%

We are ready to prove the main results of this section. This relies heavily on \cite[Thm 2.3]{HK79b} which states that any surface in $S^4$ can be unknotted by adding enough 1--handles. In particular, any pair of surfaces can be related by a sequence of 1--handle additions, connected sums with unknotted projective spaces, and their inverses. 

\begin{theorem}\label{thm:unknotting}
Let $\Tcal$ be an arbitrary multiplane diagram. There exists a finite sequence of band surgeries that turn $\Tcal$ into a multiplane diagram of an unknotted surface.   
\end{theorem}
\begin{proof}
By \cite[Thm 2.3]{HK79b}, we know that any surface in $S^4$ can be unknotted by adding enough 1--handles. Proposition \ref{prop:multisecting_tubes} ensures that each 1--handle addition can be performed via finitely many band surgeries to $\Tcal$. 
\end{proof}

\begin{theorem}\label{thm:uniqueness}
Any two multiplane diagrams are related by a finite sequence of multiplane moves, band surgeries, and their inverses. 
% Let $\Tcal_0$ and $\Tcal_1$ be two multiplane diagrams. There is a finite sequence of multiplane diagrams, starting with $\Tcal_0$ and ending with $\Tcal_1$ such that 
\end{theorem}
\begin{proof}
Let $\Tcal_1$ and $\Tcal_2$ be two multiplane diagrams representing the surfaces $F_1$ and $F_2$, respectively. By \cite[Thm 2.3]{HK79b}, there is a sequence of 1--handle additions, crosscap sums, and their inverses, taking $F_1$ into $F_2$. By Proposition \ref{prop:multisecting_tubes} and Example \ref{ex:crosscaps}, arbitrary 1--handle additions and crosscap sums can be performed using band surgeries on multiplane diagrams. Thus, we can find a finite sequence of band surgeries and their inverses, that turn $\Tcal_1$ into a multiplane diagram $\Tcal'_1$ representing the surface $F_2$. Hence, by Theorem \ref{thm:main1} there exists a finite sequence of multiplane moves turning $\Tcal'_1$ into $\Tcal_2$. 
\end{proof}

%%%%%%%%%%%%%%%%%%%%%%%%%%%%%%%%%%%%%%%%%%%

%%%%%%%%%%%%%%%%%%%%%%%%%%%%%%%%%%%%%%%%%%%%
\section{Graphs induced by bridge multisections}\label{sec:graphs}
An \emph{edge coloring} of a graph is a proper coloring of the edges, meaning an assignment of colors to edges so that no vertex is incident to two edges of the same color. An edge coloring with $n$ colors is called an \emph{$n$--edge coloring}. The spine of a bridge multisection is an $n$--valent graph with an edge coloring where the edges of color $i$ are the arcs of the $i$-th tangle $T_i$. %The spine of a bridge multisection is an example of an $n$--valent graph with an $n$--edge coloring. 
This section aims to show that all $n$--edge colored $n$--valent graphs appear this way. %Theorem BLAH states that any $n$--valent graph with a given $n$--coloring is the spine of a bridge multisection for an unknotted surface. %This section shows that any such graph is the spine of some bridge multisection of an unknotted surface in $S^4$.
We do this by studying how the graph type of the multisection spine changes under band surgery along an admissible tuple $(\rho,f)$. 

Let $\Gamma$ be an $n$--valent graph with an $n$--edge coloring. % $C:E(\Gamma)\ra \{1,\dots, n\}$. 
Given an edge $e$, we can modify $\Gamma$ by an operation similar to an edge-contraction that preserves the given $n$--edge coloring. We define the \emph{edge-compression} of $\Gamma$ along $e$, denoted by $\Gamma|e$, as follows: 
The vertex set of $\Gamma|e$ equals $V(\Gamma)\setminus \partial e$. Edges parallel to $e$ in $G$ get deleted in $G|e$. There are two types of edges for $\Gamma|e$. If an edge $f$ in $\Gamma$ is not adjacent nor parallel to $e$, then $f$ is also an edge of $\Gamma|e$ with the same color. If two distinct edges $f$ and $g$ of $\Gamma$ are adjacent to $e$ and have the same color, then we add a new edge to $\Gamma|e$ that connects the endpoints of $f$ and $g$ not in $e$ %, that is the vertices $\left(\partialf\cup \partialg\right)\setminus \partial e$; this new edge will have the same color as $f$ and $g$.
with the same color as $f$ and $g$. Figure \ref{fig:k6_1} has examples of edge-compressions. 
By construction, $\Gamma|e$ is also an $n$--valent graph with an $n$--edge coloring that agrees in the common edges of $\Gamma$ and $\Gamma|e$. 

%%%%%
\begin{figure}[h]
\centering
\includegraphics[width=.65\textwidth]{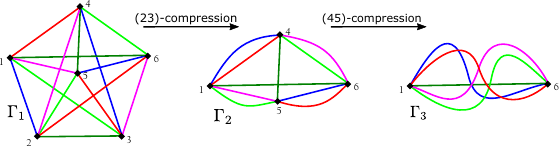}
\caption{Two edge-compressions turning the graph $K_6$ into a graph with two vertices.}
\label{fig:k6_1}
\end{figure}
%%%%%

% \begin{lemma}
% Let $\Tcal$ and $\Tcal'$ be two multiplane diagrams with spines $\Gamma$ and $\Gamma'$, respectively. Suppose that $\Tcal'=\Tcal[\rho,f]$ for some admissible tuple of arcs $(\rho,f)$. If $\rho$ has at least one type 0 arc, then $\Gamma\cong \Gamma'|e$ for some edge of $\Gamma'$. %$e\in E(\Gamma')$. 
% \end{lemma}

% \begin{proof}
% By assumption, there is a type 0 arc $\rho_{i_0}$ for some $1\leq i_0\leq n$. The tangle $T_{i_0}[\rho_{i_0}]$ is obtained from $T_{i_0}$ by adding a trivial arc around $\rho_{i_0}$. This new arc is edge $e$ of the spine of $\Gamma'$ with $\Gamma'|e\cong \Gamma$. 
% \end{proof}

As seen in \cite[Sec 3.5]{MTZ}, at the level of the 1--skeleta, edge compressions are inverses of band surgeries along admissible tuples of arcs. 
%One can see that edge-compression is the inverse operation to surgery along an admissible tuple of framed arcs. 
More precisely, %suppose that $(\rho,f)$ is an admissible tuple for a multiplane diagram $\Tcal$ with $\rho_{i_0}$ an arc of type 0 for some $i_0$. 
consider two multiplane diagrams $\Tcal_0$ and $\Tcal_1$ related by band surgery along some admissible tuple of arcs $(\rho,f)$; i.e. $\Tcal_1=\Tcal_0[\rho,f]$. If $\rho$ has at least one type 0 arc, then the skeleton of $\Tcal_0$ is isomorphic to an edge-compression on the skeleton of $\Tcal_1$. 
The following lemma shows that this relation can be reversed. 

\begin{lemma}\label{lem:building_rho}
Let $\Gamma_1$ be an $n$--valent graph with an $n$--edge coloring. Assume that there is an edge $e$ of $\Gamma_1$ and a multiplane diagram $\Tcal_0$ satisfying 
that the spine of $\Tcal_0$ with its induced $n$--edge color is isomorphic to $\Gamma_1|e$. 
%that $\Gamma|e$ is the $n$--edge colored spine of $\Tcal$. 
%$\Gamma|e=\Gamma_{\Tcal}$. 
Then there exists an admissible tuple $(\rho,f)$ for $\Tcal_0$ satisfying 
that the spine of $\Tcal_0[\rho,f]$ with its induced coloring is isomorphic to $\Gamma_1$. 
%$\Gamma = \Gamma_{\Tcal[\rho,f]}$. 
\end{lemma}
\begin{proof}
In what follows, we will recursively build $\{(\rho_{i+1},f_{i+1})\}_{i=1}^n$. Without loss of generality, we assume that the color of $e$ is $1$. This way, $(\rho_1,f_1)$ can be chosen to be any type 0 framed arc in $T_1$. Note that, regardless of our final choice for $(\rho_n,f_n)$, the pair $(\rho_n,\rho_1)$ will satisfy the admissibility conditions of Definition \ref{def:admissible_arcs}(3-4) as $\rho_1$ is already type 0. 
Assume that the arc $(\rho_i,f_i)$ has been defined for some $i=1,\dots, n-1$. In order to choose $(\rho_{i+1},f_{i+1})$, we look at the $(i+1)$--colored edges in $\Gamma$ that are adjacent to $e$. We have two cases, depending on how many such edges are. 
If there is only one such edge, this edge is parallel to $e$, and we can then choose $\rho_{i+1}$ to be a type 0 arc in $T_{i+1}$ with any framing $f_{i+1}$. 

Suppose there exist two $(i+1)$--colored edges in $\Gamma$ adjacent to $e$ denoted by $g$ and $h$. By definition of $\Gamma|e$, the edges $g$ and $h$ get traded with a new $(i+1)$--colored edge connecting the vertices in $(\partial g \cup \partial h) \setminus \partial e$. As $\Gamma|e$ is the spine of $\Tcal_0$, this new edge corresponds to a strand $t_{i+1}$ of $T_{i+1}$. %So $\rho_{i+1}$ will be a type 1 arc of $T_{i+1}$ with one endpoint in $t_{i+1}$. 
If $\rho_i$ is type 0, 
% any type 1 arc with one endpoint in $t_{i+1}$ will do for $\rho_{i+1}$. 
we can choose $\rho_{i+1}$ to be any type 1 arc with one endpoint in $t_{i+1}$ and the admissibility conditions of Definition \ref{def:admissible_arcs}(3-4) will be satisfied for $\rho_i\cup\rho_{i+1}$. 
The framing of $\rho_{i+1}$ will be chosen so that $T_{i+1}[\rho_{i+1},f_{i+1}]$ connects the correct endpoints of $g$ and $h$. 

We are left to discuss the case when $\rho_i$ is a type 1 arc. The following argument takes place in a given diagram for the bridge splitting $(S^3,L_i)=(B_i,T_i)\cup_\Sigma (B_{i+1},T_{i+1})$. 
First, slide the endpoints of $\rho_i$ so that this arc is inside a small neighborhood of the bridge sphere, $N(\Sigma)$. By a result of Solanki \cite[Cor 2]{solanki_plats_unlink}, there is a sequence of braid isotopies, pocket moves, flip moves, and bridge destabilizations taking $L_i=T_i\cup T_{i+1}$ to a crossingless diagram of $L_i$ in with $|L_i|$ bridges. See \cite[Sec 2.2]{solanki_plats_unlink} for detailed definitions of these moves. We keep track of $\rho_i$ through this sequence and make sure that $\rho_i$ is still a decreasing arc lying in $N(\Sigma)$. 
Note that braid isotopies and flip moves preserve the desired properties of $\rho_i$ as they preserve the regular levels of the Morse function associated with the genus zero Heegaard splitting of $S^3$. 
Pocket moves can be pre- and post-composed with braid isotopies so that they are supported away from $N(\Sigma)$, thus preserving the decreasing property of $\rho_i$. 
Every time we find a destabilization move in our sequence, instead of canceling the pair of critical points, we can push the $(\min,\max)$ pair to be destabilized to lie inside $N(\Sigma)$ as in Figure \ref{fig:building_rho}. Then, future moves in our sequence can be replicated for this $b$--bridge presentation. 

%%%%%
\begin{figure}[h]
\centering
\includegraphics[width=.5\textwidth]{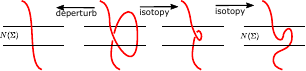}
\caption{How to isotope $L_i$ (right) instead of performing a destabilization (left).}
\label{fig:building_rho}
\end{figure}
%%%%%

At the end of the above process, we are left with a $b$--bridge crossingless diagram for $L_i$ and a type 1 arc $\rho_i$ that braids with some strands of $L_i$. 
We can then choose $\rho_{i+1}$ so that it connects $\{*\}=\rho_i\cap int(\Sigma)$ with our strand $t_{i+1}$ while undoing all the braiding caused by $\rho_i$. This way, $\rho_i\cup \rho_{i+1}$ can be isotoped (via Reidemeister II moves) to have no crossings with $L_i$. Thus, $\rho_i\cup \rho_{i+1}$ satisfies the admissibility conditions in Definition \ref{def:admissible_arcs}(3-4). To end, we pick the framing $f_{i+1}$ so that $f_i+f_{i+1}$ is zero or $\pm 1$, depending on which endpoint of $e$ (puncture of $\Sigma$) is connected by the arcs of $t_{i+1}[\rho_{i+1},f_{i+1}]$. 
This concludes the construction of $(\rho,f)$. 
\end{proof}

We are ready to prove the main theorem of this section. The idea of the proof is to take a sequence of edge-compressions turning a given graph into a graph with two edges, and then undo the edge-compressions using Lemma \ref{lem:building_rho}. Figures \ref{fig:k6_1} and \ref{fig:k6_2} are the result of running this procedure on the complete graph in six vertices $K_6$. 

%%%%%%
%\begin{figure}[h]
%\centering
%\includegraphics[width=.85\textwidth]{images/fig_k6_0.PNG}
%\caption{words }
%\label{fig:k6_0}
%\end{figure}
%%%%%%

%%%%%
\begin{figure}[ht]
\centering
\includegraphics[width=.9\textwidth]{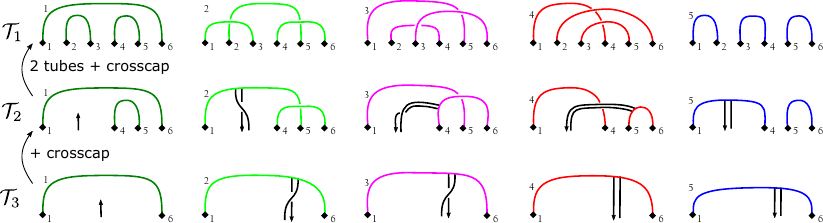}
\caption{Guided by the compressions in Figure \ref{fig:k6_1}, we found admissible tuples of arcs turning the 1--bridge 5--section of an unknotted 2--sphere into a 3--bridge 5--section with spine isomorphic to $K_6$. One can check how the band surgeries change the multisected surface using the reduction rules on Remark \ref{rem:counting_tubes}; see Example~\ref{fig:fig25}.}
\label{fig:k6_2}
\end{figure}
%%%%%

\begin{theorem}\label{thm:Tait_spines}
Let $\Gamma$ be an $n$--valent graph with an $n$--edge coloring. There is a bridge multisection $\Tcal$ for an unknotted surface in $S^4$ with 1--skeleton isomorphic to $\Gamma$ and inducing the given edge coloring of $\Gamma$. 
\end{theorem}

\begin{proof}
We only discuss the case when $\Gamma$ is connected. The disconnected case will follow since the spine of the distant sum of bridge multisections is the disjoint union of their respective spines (see Section \ref{sec:sums}). 
Let $\{e_1,\dots, e_b\}$ be the edges of $\Gamma$ with color $1$. Recursively, define $\Gamma_1=\Gamma$ and $\Gamma_{k+1}=\Gamma_k|e_k$ for $k=1,\dots, b-1$. By construction, $\Gamma_b$ has two vertices and $n$ parallel edges of distinct colors. Thus $\Gamma_b$ is the spine of a \mbox{1--bridge} $n$--section of an unknotted 2--sphere. Denote this multiplane diagram by $\Tcal_b$. We now apply Lemma \ref{lem:building_rho} multiple times to get a sequence of admissible tuples $(\rho^1,f^1),\dots, (\rho^{b-1},f^{b-1})$ undoing the edge-compressions $\Gamma_1\rightarrow \cdots \rightarrow \Gamma_b$. More precisely, they satisfy that 
\[ \Tcal_{b-j}=\Tcal_{b-j+1}[\rho^j,f^j] \quad \text{and} \quad 
%\Gamma_{\Tcal_{b-j}}=\Gamma_{b-j},\]
\Gamma_{b-j} \text{ is the spine of } \Tcal_{b-j},\]
for every $j=1,\dots, b-1$. In particular, $\Tcal_1$ satisfies that $\Gamma=\Gamma_1$ is the spine of $\Tcal_1$
By Theorem \ref{thm:surgery_rhos_tubing}, the surface represented by $\Tcal_1$ is obtained by a sequence of 1--handles and crosscaps added to an unknotted \mbox{2--sphere.} By Lemma \ref{lem:tubing_preserves_unknot}, these operations (applied to connected surfaces) preserve the property of being unknotted, thus $F_{\Tcal_1}$ is still unknotted. 
\end{proof}

The proof of the following result is the same as that of \cite[Cor 1.2]{MTZ}.
\begin{corollary}
    Every multiplane diagram $\mathcal{P}$ of a knotted surface $\mathcal{K}\subset S^4$ can be converted to a multiplane diagram $\mathcal{P}'$ of an unknotted surface $\mathcal{U}\subset S^4$ after a sequence of interior Reidemeister moves and crossing changes.
\end{corollary}

\begin{remark}
At this point, the invested reader may wonder if Theorem~\ref{thm:Tait_spines} follows from its $n=3$ version proven in \cite[Thm 1.1]{MTZ}; that every 3--valent graph with a 3--edge coloring is the spine of an unlink of unknotted surfaces. Below, we describe an inductive approach one may try that fails. Thus, this motivates the use of our technology in Lemma~\ref{lem:building_rho}. That said, there may still be an alternative proof of Theorem~\ref{thm:Tait_spines} unknown to the authors of this work. 

First, observe that removing all edges with a fixed color from a $4$-colored graph yields a $3$-valent graph with a $3$-edge coloring. So, by \cite{MTZ}, every 4-colored graph is the spine of a bridge 4-section $\Tcal=(T_1,T_2,T_3,T_4)$ for $F$ where $\Tcal'=(T_1,T_2,T_3)$ and $\Tcal''=(T_3,T_4,T_1)$ are spines for unlinks for unknotted surfaces $F'$ and $F''$, respectively. It is not hard to see that $F$ is built from $F'$ and $F''$ by removing disks bounded by $T_3\cup \T_1$ from both $F'$ and $F''$ and gluing the resulting surfaces. In other words, $F$ is obtained from $F\sqcup F''$ by adding $|T_1\cup \T_3|$ 1--handles between the two surface-links. As both $F'$ and $F''$ are unknotted, one may want to show that $F$ is also unknotted. This is not always the case; e.g., the spun trefoil can be described by one 1--handle addition to the distant sum of two unknotted spheres as in Figure~\ref{fig:exam_tube_2}. Another comment is that in the same figure, the tuples of tangles (blue, green, purple) and (purple, red, blue) are triplane diagrams for unlinks of 2--spheres while the 4--sected surface is knotted. 
\end{remark}

%%%%%%%%%%%%%%%%%%%%%%%%%%%%%%%%%%%%%%%%%%%%%
\bibliographystyle{alpha}
\bibliography{sources}
\end{document}